\documentclass{article}

\RequirePackage[OT1]{fontenc}
\RequirePackage[numbers]{natbib}
\RequirePackage[colorlinks,citecolor=blue,urlcolor=blue]{hyperref}\usepackage[latin1]{inputenc}
\usepackage{amsmath,amsfonts,amssymb,amsopn,amscd}
\usepackage{bbm,wasysym,stmaryrd}
\newcommand{\R}{\mathbbm{R}}

\newcommand{\Z}{\mathbbm{Z}}

\newcommand{\Exp}{\mathbb{E}} 

\newcommand{\T}{\mathcal{T}}  

\newcommand{\snorm}[1]{\left| #1 \right|_{\mathcal{E}}}

\newcommand{\mcY}{\mathcal{Y}}
\newcommand{\mcX}{\mathcal{X}}

\newcommand{\C}{\mathcal{C}}

\newcommand{\norm}[1]{\left\| #1 \right\|}
\newcommand{\lsup}[1]{\underset{#1\to\infty}{\overline{\lim}}}
\newcommand{\linf}[1]{\underset{#1\to\infty}{\underline{\lim}}}

\newcommand{\modd}{\text{ mod }}

\newtheorem{theorem}{Theorem}
\newtheorem{corollary}[theorem]{Corollary}

\newtheorem{lemma}[theorem]{Lemma}

\newtheorem{assumption}{Assumption}

\newenvironment{proof}{\paragraph{Proof:}}{\hfill$\square$}

\begin{document}
\title{Large Deviations of a Network of Neurons with Dynamic Sparse Random Connections}
\markboth{MacLaurin}{Large Deviations of a Network of Neurons with Dynamic Sparse Random Connections}

\author{James  MacLaurin}

\maketitle
\footnote{Department of Mathematical Sciences, New Jersey Institute of Technology, MLK Boulevard Newark USA\\
james.n.maclaurin@njit.edu}


 \begin{abstract}
In this work we determine the limiting dynamics of a system of interacting particles indexed by a lattice $\Z^d$. The connections are random, sparse and unscaled, so that the system converges in the large size limit due to the probability of a connection between any two particles decreasing as the system size increases. The particles are also subject to noise (such as independent Brownian Motions). The method of proof is to assume a process-level (or Level 3) LDP for the double-layer empirical measure for the noise and connections, and then apply a series of transformations to this to obtain an LDP for the process-level empirical measure of our system. Although it is not explicitly necessary, we expect that most applications of this work should involve an assumption of stationarity of the probability law for the noise and connections under translations of the lattice, so that the system converges to an ergodic probability law in the large size limit. This work synthesizes the theory of large-size limits of interacting particles with that of random graphs and matrices. It should therefore be relevant to neuroscience and social networks theory in particular.
\end{abstract}
 
\section{Introduction}
In this paper we determine a Large Deviation Principle for an asymptotically large system of interacting processes on a lattice with sparse random connections. We are motivated in particular by the study of interacting neurons in neuroscience \cite{bressloff:12}, but this work ought to be adaptable to other phenomena such as mathematical finance, social networks, population genetics or insect swarms \cite{robins2007introduction,atay2016perspectives}. 

Classical mean-field models are perhaps the most common method used to scale up from the level of individual particles to the level of populations \cite{bossy-talay:97,baladron-fasoli-etal:12b}. For a group of neurons indexed from $1$ to $N$, the evolution equation of a mean field model is typically of the following form (an $\R^N$-valued SDE)
\begin{equation}
dX^j_t = \left[g(X^j_t) + \frac{1}{N}\sum_{k=1}^N h_t(X^j,X^k)\right]dt + \sigma(X^j_t)dW^j_t.\label{eq:meanfieldbasic}
\end{equation}
We set $X^j_0= 0$. Here $g$ is Lipschitz, $h$ is Lipschitz and bounded, and $\sigma$ is Lipschitz. $(W^j)_{j\in\mathbb{N}}$ are independent Brownian Motions representing internal / external noise. Asymptoting the number of particles $N$ to $\infty$, we find that in the limit $X^j$ is independent of $X^k$ (for $j\neq k$), and each $X^j$ is governed by the same law \cite{sznitman:91}. Since the $(X^j)_{j\in\mathbb{N}}$ become more and more independent, it is meaningful to talk of their mean as being representative of the group as a whole. In reaching this limit, three crucial assumptions have been made: that the external noise is uncorrelated, that the connections between the particles are homogeneous and that the connections are scaled by the inverse of the size of the system. We will relax each of these assumptions in our model (which is outlined in Section \ref{sect:NeuronModel}).

The major difference between the model in Section \ref{sect:NeuronModel} and the mean field model outlined above is the model of the connections. In the mean field model, the network is completely connected, with a uniform connection strength of $N^{-1}$. In our work, the connections are sparse and random, and typically (although not necessarily) sampled from a spatially-stationary probability law (i.e. a law that is invariant under translations of the lattice). The system does not blowup as $N\to \infty$ because the probability of a connection between particles decreases as the lattice distance increases. One reason that the mean-field model in \eqref{eq:meanfieldbasic} converges as $N\to\infty$ is because the probability law is invariant under any permutation of the indices. By contrast, the system of this paper is only invariant under a translation of the lattice (with periodic boundary conditions), a much weaker assumption that can accommodate a richer diversity of emergent phenomena.

Our network has the structure of a sparse random graph with edges which are directed and weighted. In neuroscience and social network theory in particular, there is an enormous literature devoted to the study of complex networks \cite{sporns2011networks,robins2007introduction}.  In the mathematics literature, the papers \cite{delattre2016hawkes,delattre2016note,coppini2020law,bayraktar2020graphon} determine the limiting behavior for dynamics on Erdos-Renyi random graphs where the typical number of afferent edges on each node diverges with the system size. In the large $N$ limit, the behavior of such networks becomes analogous to the mean-field equations in \eqref{eq:meanfieldbasic}. By contrast, in the model of this paper, as the network grows, the typical number of afferent edges per neuron does not diverge with $N$. In the model of this paper, the connection strength between any two vertices is random, such that the probability of a strong connection decreases as the lattice distance increases. Furthermore the connections are correlated, with the correlation between two connections dependent on the lattice distance between the heads and tails.  We are able to obtain an understanding of the limiting behavior because the joint law of the dynamics and edges is invariant under a translation of the lattice. This means that instead of working with the standard empirical measure, we must work with the `Level 3' empirical process that understands correlations between different particles \cite{ellis:85}.

The main result of this paper is a Large Deviation Principle (LDP) for the interacting particle model in Section \ref{sect:NeuronModel}. A Large Deviation Principle is a very useful mathematical technique that allows us to estimate finite-size deviations of the system from its limit behaviour. There has been much effort in recent years to understand such finite-size phenomena in mathematical models of neural networks - see for instance \cite{bressloff:09,buice-cowan-etal:10,fasoli2015complexity}. More generally, there exists a well-developed literature on the Large Deviations and other asymptotics of weakly-interacting particle systems (see for example \cite{ben-arous-guionnet:95,goldys:01,dawson-del-moral:05,budhiraja-dupuis-etal:12,delarue2015particle,faugeras-maclaurin:16,bossy2015clarification}). These are systems of $N$ particles, each evolving stochastically, and usually only interacting via the empirical measure. 

This paper is structured as follows. In Section \ref{sect:stochprocess} we outline a general model of interacting particles on a lattice with random connections, and state a large deviation principle under a set of assumptions.  In Section \ref{Section Application} we outline an extended example of this theory which satisfies the assumptions of Section \ref{sect:stochprocess}. This example considers a sparse network of Fitzhugh-Nagumo neurons, with Hebbian learning on the connections, and the probability of a connection between two neurons being given by a Gibbs distribution. The remaining sections are dedicated to the proof of the main result in Theorem \ref{Theorem Main LDP}. The proof essential demonstrates that the empirical process can be approximated arbitrarily well by a continuous filtering of the empirical process containing the connections and noise.

\textit{Notation:}

If $\mathcal{X}$ is some separable topological space, then we denote the $\sigma$-algebra generated by the open sets by $\mathcal{B}(\mathcal{X})$, and the set of all probability measures on $(\mathcal{X},\mathcal{B}(\mathcal{X}))$ by $\mathcal{P}(\mathcal{X})$. We always endow $\mathcal{P}(\mathcal{X})$ with the topology of weak convergence. 

Elements of the processes in this paper are indexed by the lattice points $\Z^d$: for $j\in\Z^d$ we write $j = (j(1),\ldots,j(d))$. Let the particles be indexed by the set $V_n \subset \Z^d$, which is such that $j\in V_n$ if $|j(m)| \leq n $ for all $1\leq m \leq d$. The number of elements in $V_n$ is written as $|V_n| := (2n+1)^d$.

Let $| X | $ denote the Euclidean norm on $\mathbb{R}^m$. For any $s\in [0,T]$, we endow $\mathcal{C}([0,s],\R^m)$ with the norm $\norm{U}_s := \sup_{r\in [0,s]}|U_r|$, and we write $\mathcal{T} := \mathcal{C}\big([0,T],\R^m\big)$. In the model outlined in the next section, the activity over time of each particle is a $\mathcal{T}$-valued random variable. The state space for the connections between the particles is taken to be a complete separable metric space $\mathcal{E}$, with metric $d_{\mathcal{E}}(\cdot,\cdot)$. Let $\pi^{V_m}: \T^{\Z^d}\to \T^{V_m}$ be the projection $\pi^{V_m}(X) := (X^j)_{j\in V_m}$. We endow $\T^{\Z^d}$ with the cylindrical topology (generated by sets $O\subset \T^{\Z^d}$ such that $\pi^{V_m}O$ is open in $\T^{V_m}$).  Let $d_m^{\mathcal{P}}$ be the Levy-Prokhorov metric on $\mathcal{P}\big(\T^{V_m}\big)$ generated by the norm $\norm{x}_{T,m} := \sum_{j\in V_m}\norm{x^j}_{T}$ on $\T^{V_m}$. The following metric on $\mathcal{P}\big(\T^{\Z^d}\big)$ metrizes weak convergence,
\begin{equation}\label{defn prohorov metric T}
d^{\mathcal{P}}(\mu,\nu) := \sum_{j=1}^\infty \min \bigg(2^{-j},d^{\mathcal{P}}_j\big(\pi^{\mathcal{P}}_{V_j}\mu,\pi^{\mathcal{P}}_{V_j}\nu\big) \bigg),
\end{equation}
where $\pi^{\mathcal{P}}_{V_j}$ is the projection onto the marginal in $\mathcal{P}\big(\T^{V_j} \big)$. Thanks to the Komolgorov Extension Theorem, $d^{\mathcal{P}}$ is complete (i.e. each Cauchy Sequence converges to a unique limit), since each $d_j^{\mathcal{P}}$ is complete over $\mathcal{P}\big(\T^{V_j}\big)$ \cite{billingsley:99}. 

It may be noted that many papers on the convergence of networks of interacting particles use the Wasserstein Metric rather than the above Levy Prokhorov metric. However the implementation of the Wasserstein metric is complicated by the fact that the map $\Psi^m$ (which maps the noise and connections to the solution and is defined in \eqref{eq: first Psi m definition}) is not continuous with respect to the cylindrical topology. This is why we use the Levy-Prokhorov metric in this paper.
\section{Outline of Model and Preliminary Definitions}\label{sect:stochprocess}
In this section we outline our finite model of $(2n+1)^d$ stationary interacting particles indexed over the cube $V_n$. In Subsection \ref{Section Assumptions} we outline our assumptions on the model. The main result of this paper is stated in Theorem \ref{Theorem Main LDP}. 

\subsection{Outline of Model and Main Result}\label{sect:NeuronModel}

For $n\in\Z^+$, there are $|V_n|$ particles in our network. The joint probability law (of both the dynamics and the connections) is assumed to be invariant under a translation of the lattice (assuming periodic boundary conditions).  There are three components to the dynamics of our network: the internal dynamics term $\mathfrak{b}$, the interaction term $\Lambda^k_s(J^{n,j,k},U^j,U^{(j+k)\modd V_n})$ and the noise term $W^{n,j}_t$. The form of our interaction term differs from standard mean-field models in that it is not scaled by some function of the network size $|V_n|$, and it is not homogeneous throughout the network. It depends also on the random connection $J^{n,j,k}$ between the particles at nodes $j$ and $(j+k)\modd V_n$ (here $k \in V_n$ is the vector distance between the two particles relative to the toroidal topology). $(J^{n,j,k})_{j,k\in V_n}$ are correlated random variables taking values in $\mathcal{E} := \lbrace 0,1\rbrace$. 

We assume that $W^{n} := \big(W^{n,j}_{t}\big)_{j\in V_n,t\in [0,T]}$ is a $\mathcal{T}^{V_n}$-valued random variable (such as, but not necessarily, independent Brownian motions). The system we study in this paper is governed by the following evolution equation: for $j\in V_n$ and $t\in [0,T]$,
\begin{equation}\label{eq:fundamentalmult}
U^{j}_t = \int_0^t \bigg(\mathfrak{b}(U^{j}_s) + \sum_{k\in V_n}\Lambda^{k}_s\big(J^{n,j,k},U^j,U^{(j+k)\modd V_n}\big)\bigg)ds + W^{n,j}_{t}.
\end{equation}

Here $(j+k)\modd V_n := l\in V_n$, such that $(j(p)+k(p)) \modd (2n+1) = l(p)$ for all $1\leq p\leq d$. Thus one may think of the particles as existing on a $d$-dimensional torus. It is noted in Lemma \ref{Lemma Solution Equivalence} that there exists a unique solution to \eqref{eq:fundamentalmult} for every $W^n \in \T^{V_n}$. The thrust of this article is to understand the asymptotic behaviour of the network as $n\to \infty$.

Because the probability of a connection decreases with lattice distance, the network is very sparse. This means that the standard empirical measure in $\mathcal{P}\big( \mathcal{C}([0,T],\mathbb{R}^m)\big)$ is not the appropriate object to capture the emergent behavior.  Rather, because the system is invariant under a translation of the lattice, the process-level empirical measure $\hat{\mu}^n$ is a more appropriate limiting object. This is a reduced macroscopic variable commonly used in statistical mechanics  \cite{ellis:85}. Let $S^k:\T^{\Z^d} \to \T^{\Z^d}$ (for some $k\in\Z^d$) be the shift operator (i.e. $(S^k x)^m := x^{m+k}$). Let $\mathcal{P}_S(\T^{\Z^d})$ be the set of all stationary probability measures, i.e. such that for all $k\in \Z^d$,
\[
\mu \circ (S^k)^{-1} = \mu.
\]
Denote the empirical measure $\hat{\mu}^n: \T^{V_n} \to \mathcal{P}_S(\T^{\Z^d})$ by
\begin{equation}
\label{defn empirical measure}
\hat{\mu}^n(X) := \frac{1}{|V_n|}\sum_{j\in V_n}\delta_{S^j \tilde{X}},
\end{equation}
where $\tilde{X}\in\T^{\Z^d}$ is the $V_n$-periodic interpolant, i.e. $\tilde{X}^j := \tilde{X}^{j \modd V_n}$. If $X \in \T^{\Z^d}$, then in a slight abuse of notation we write $\hat{\mu}^n(X) := \hat{\mu}^n(\pi^{V_n}X)$. 

We now outline the main result of this paper.
\begin{theorem}\label{Theorem Main LDP}
Let the law of $\hat{\mu}^n(U)$ be $\Pi^n \in \mathcal{P}(\mathcal{P}_S(\T^{\Z^d}))$. Under the assumptions outlined in Section \ref{Section Assumptions}, $(\Pi^n)_{n\in\Z^+}$ satisfy a Large Deviation Principle with good rate function $I$ (i.e. $I$ has compact level sets). That is, for all closed subsets $A$ of $\mathcal{P}_S\big(\T^{\Z^d}\big)$,
\begin{equation}\label{LDP closed}
\lsup{n}\frac{1}{|V_n|}\log\Pi^n(A) \leq - \inf_{\gamma\in A}I(\gamma).
\end{equation}
For all open subsets $O$ of $\mathcal{P}_S\big(\T^{\Z^d}\big)$,
\begin{equation}
\linf{n}\frac{1}{|V_n|}\log\Pi^n(O) \geq - \inf_{\gamma\in O}I(\gamma).\label{LDP open}
\end{equation}
The rate function is
\[
I(\mu) := \inf_{\nu \in \mathcal{P}_{\bar{S}}(\bar{\T}^{\Z^d})}\left\lbrace I_{\mathcal{Y}}(\nu): \grave{\Psi}(\nu) = \mu\right\rbrace,
\]
where $\mathcal{P}_{\bar{S}}(\bar{\T}^{\Z^d})$ is defined in \eqref{eq bar S stationary}, $\grave{\Psi}: \mathcal{P}_{\bar{S}}(\bar{\T}^{\Z^d}) \to \mathcal{P}_{\bar{S}}(\bar{\T}^{\Z^d})$ is defined in \eqref{eq:fundamentalmult 1},\eqref{eq: grave Psi 1},\eqref{eq: grave Psi 2} and $I_{\mathcal{Y}}$ is defined in Assumption \ref{Assumption LDP}.
\end{theorem}
From now on, if a sequence of probability laws satisfies \eqref{LDP closed}  and \eqref{LDP open} for some $I$ with compact level sets, then to economise space we say that it satisfies an LDP with a good rate function $I$. Theorem \ref{Theorem Main LDP} is useful because it allows us to understand the asymptotic behaviour of averages over the entire network, as is noted in the following corollary.
\begin{corollary}\label{Cor: Convergence Averages}
Suppose that there is a unique $\mu_{\mathcal{Y}} \in \mathcal{P}_{\bar{S}}\big(\bar{T}^{\Z^d} \big)$ such that $I_{\mathcal{Y}}\big(\mu_{\mathcal{Y}}\big) = 0$ (see Assumption \ref{Assumption LDP} for the definitions of these terms). For some positive integer $q$, let $g:\T^{V_q} \to \R$ be continuous and bounded. Then, writing $H_n = \frac{1}{|V_n|}\sum_{j\in V_n}g\big(S^j \cdot\tilde{U}\big)$,
\[
\lim_{n\to \infty}H_n = \Exp^{\grave{\Psi}(\mu_{\mathcal{Y}})}\left[ g\right],
\]
$\mathbb{P}$-almost-surely.
\end{corollary}
\subsection{Assumptions}\label{Section Assumptions}

We employ the following assumptions. 

In many interacting particle models, such as the Fitzhugh-Nagumo model in Section \ref{Section Application}, the internal dynamics term $\mathfrak{b}$ is not Lipschitz. In particular, $\mathfrak{b}$ is usually strongly decaying when the activity is greatly elevated, so that $\mathfrak{b}$ always acts to restore the particle to its resting state. This decay is necessary in order for the neurons to exhibit their characteristic `spiking' behaviour. The following assumptions can accommodate this non-Lipschitz behaviour.
\begin{assumption}\label{Assumption Absolute Bound}
The following one-sided growth and Lipschitz conditions are assumed. It is assumed that there exists a positive constant $C$ such that for all $\mathbf{z},\mathbf{y} \in \mathbb{R}^m$, and all $1\leq p \leq m$, writing $\mathfrak{b} = (\mathfrak{b}_1,\ldots,\mathfrak{b}_m)$,
\begin{align*}
\mathfrak{b}_p(\mathbf{z}) &\leq C | \mathbf{z} | \\
-\mathfrak{b}_p(\mathbf{z}) &\leq C | \mathbf{z} | \\
 \mathfrak{b}_p(\mathbf{z}) - \mathfrak{b}_p(\mathbf{y}) &\leq C | \mathbf{z} - \mathbf{y} | \\
  \mathfrak{b}_p(\mathbf{y}) - \mathfrak{b}_p(\mathbf{z}) &\leq C | \mathbf{z} - \mathbf{y} | .
\end{align*}
\end{assumption}
 This paper is primarily intended to model sparse networks. The system converges, not because there are $O(|V_n|)$ connections which decrease in magnitude either uniformly as the network grows (as in mean-field models) or as the lattice distance increases (as in \cite{faugeras-maclaurin:16}), but rather because the probability of there being a connection between two particles decreases as the lattice distance increases. We are primarily concerned with models where, if there is a connection between two particles, then it is strong, even if they are a long way apart on the lattice. 
The interactions $\Lambda^k_s(\cdot,\cdot,\cdot)$ in many of the networks that we want to study are typically nonlinear. See for example the model of the interactions in \cite{baladron-fasoli-etal:12b}, and also the example model in Section \ref{Section Application}. 
\begin{assumption}\label{Assumption NonUniform}
We assume that if there is no connection, the interaction is identically $0$, i.e.
 \begin{equation}
 \Lambda^k_s(0,\cdot,\cdot) = 0.
 \end{equation}
Assume that for all $k\in \Z^d$, $\Lambda^k_s(\cdot,\cdot,\cdot)$ is continuous on $[0,T]\times\mathcal{E}\times \T \times \T$, and for each $s\in [0,T]$, $\Lambda^k_s$ is $\mathcal{B}(\R) / \mathcal{B}(\mathcal{E})\times\mathcal{B}\big(\mathcal{C}([0,s],\R)\big)\times \mathcal{B}\big(\mathcal{C}([0,s],\R)\big)$-measurable. For all $x,y\in \mathcal{E}$, $t\in [0,T]$, $U,X,Z \in \T$ and $k\in\Z^d$
\begin{align*} 
\big|\Lambda^{k}_t(x,U,X) - \Lambda_t^{k}(x,U,Z)\big| &\leq  \norm{X-Z}_t,\\
\big| \Lambda^{k}_t(x,U,X) - \Lambda_t^{k}(x,Z,X)\big| &\leq  \norm{U-Z}_t,\\
\big| \Lambda^{k}_t(x,U,X) - \Lambda_t^{k}(y,U,X)\big| &\leq \big(\norm{U}_t + \norm{X}_t \big) |x-y|.
\end{align*}
We also assume that
\begin{equation}\label{eq Lambda abs bound}
\big|\Lambda^k_t(x,X,Z)\big| \leq \big(1+\norm{X}_t \big).
\end{equation}
\end{assumption}
Since, in contrast to the great majority of papers on interacting particles surveyed in the introduction, the magnitude of the interactions is independent of the network size, the effect of the interactions grows rapidly as the entire system grows. The following bound ensures that the system does not blowup as the size scales to infinity.  
\begin{assumption}\label{Assumption Connection Growth}
We assume that there exist positive constants $\mathfrak{a}_1,\mathfrak{a}_2 > 0$, a constant $\rho > 1$ and a positive integer $\mathfrak{m}_0$ such that for all $m \geq \mathfrak{m}_0$,
\begin{align*}\label{defn Ac}
\Exp\bigg[\exp&\bigg(\mathfrak{a}_1|V_m|^{2\rho+2}\exp\big((4+2\rho)TC_J|V_m| \big)\sum_{j\in V_n}\big(\sum_{k\notin V_m}\snorm{J^{n,j,k}} \big)^2\bigg)\bigg] \\  &\leq \exp\big(\mathfrak{a}_2 |V_n| \big)\\
\Exp\bigg[\exp&\bigg(\mathfrak{a}_1|V_m|^{\rho+1}\exp\big((3+\rho)TC_J|V_m| \big)\sum_{j\in V_n}\sum_{k\notin V_m}\snorm{J^{n,j,k}} \bigg)\bigg] \\ &\leq \exp\big(\mathfrak{a}_2 |V_n| \big)
\end{align*}
\end{assumption}

The heart of the method in this paper is to recognize that the solution $U$ of \eqref{eq:fundamentalmult} can be written as an ergodic transformation in an ambient higher-dimensional space $\bar{\T}^{\Z^d}$ containing both the noise and the connections. The Large Deviation Principle for $U$ can then be obtained from the (assumed) Large Deviation Principle of $W^n$ and $J^n$ using transformation methods, since the empirical measure for $U$ can be written as a function of the \textit{double layer} empirical measure which incorporates both $W^n$ and $J^n$ (see Lemma \ref{Lemma Solution Equivalence}). 

We now define the ambient space $\bar{\T}^{\Z^{d}}$ containing both the connections and the disorder. Let $\bar{\mathcal{T}} = \mathcal{T}\times \mathcal{E}^{\Z^d}$, with $\mathcal{E}^{\Z^d}$ endowed with the cylindrical topology and $\bar{\mathcal{T}}$ with the product topology. Note that $\bar{\mathcal{T}}$ is separable. For $l\in \Z^d$, let $\bar{S}^l: \bar{\T}^{\Z^{d}} \to \bar{\T}^{\Z^{d}}$ be $(\bar{S}^l\cdot \mcX)^{j} := \mcX^{j+l}$. Let $\mathcal{P}_{\bar{S}}\big(\bar{\T}^{\Z^{d}} \big)$ be the set of all stationary measures, i.e. such that $\mu \in \mathcal{P}_{\bar{S}}\big(\bar{\T}^{\Z^{d}} \big)$ if and only if for all $k\in\Z^{d}$,
\begin{equation}
\mu\circ(\bar{S}^k)^{-1} = \mu.\label{eq bar S stationary}
\end{equation}
For $j\in V_n$ and $k\notin V_n$, let $J^{n,j,k} = 0$ (note that this doesn't affect the dynamics in \eqref{eq:fundamentalmult}). Let $\tilde{\mcY}^n$ be the $\bar{\T}^{\Z^{d}}$-valued random variable such that $\tilde{\mcY}^{n,p} := \big(W^{n,p\modd V_n}, \lbrace J^{n,p\modd V_n,k}\rbrace_{k\in\Z^d}\big)$. Let the law of 
\begin{equation}\label{defn double layer empirical measure}
\hat{\mu}^n(\mathcal{Y}^n) := \frac{1}{|V_n|}\sum_{k\in V_n}\delta_{\bar{S}^k \tilde{\mcY}^n}
\end{equation}
be $\Pi^n_{\mcY} \in \mathcal{P}\big(\mathcal{P}_{\bar{S}}(\bar{\T}^{\Z^{d}})\big)$. $\hat{\mu}^n(\mathcal{Y}^n)$ is sometimes known as the \textit{double-layer} empirical measure, as it incorporates both the noise and the random environment \cite{dai1996mckean}.

\begin{assumption}\label{Assumption LDP}
The series of laws $\big(\Pi^n_{\mathcal{Y}}\big)_{n\in \Z^+}$ is assumed to satisfy a Large Deviation Principle with good rate function $I_{\mathcal{Y}} :\mathcal{P}_{\bar{S}}\big(\bar{\T}^{\Z^d}\big) \to \R^+$. 
\end{assumption}
 In many proofs of level-3 Large Deviations results, one starts from an i.i.d process, and applies standard transformation methods (either an exponential change of measure, as with LDPs for Gibbs distributions \cite{georgii:88}, or a moving average transformation \cite{donsker-varadhan:85,e16126722,faugeras-maclaurin:16}). If $W^n$ is independent of $J^n$, and the empirical measures of $\hat{\mu}^n(W^n)$ and $\hat{\mu}^n(J^n)$ each satisfy LDPs obtained through these transformations, then generally $\hat{\mu}^n(\mathcal{Y}^n)$ will satisfy an LDP as well. An example of this is given in Section \ref{Section Application}.
\begin{assumption}\label{Assumption Noise LDP}
It is assumed that there exist positive constants $\mathfrak{e}_1,\mathfrak{e}_2$ such that
\begin{equation}\label{eq: a limit}
\lsup{n}\frac{1}{|V_n|}\log \Exp\bigg[\exp\bigg(\mathfrak{e}_1\sum_{j\in V_n}\norm{W^{n,j}}_T^2 \bigg) \bigg] \leq \exp\big(|V_n|\mathfrak{e}_2\big).
\end{equation}
\end{assumption}

\section{An Example Application: A Fitzhugh-Nagumo Neural Network with Gibbsian Random Connectivity and Learning}\label{Section Application}

In this section we outline an example from neuroscience of a model satisfying \eqref{eq:fundamentalmult} and the assumptions of Section \ref{Section Assumptions}. Most of these assumptions are satisfied relatively easily, and we therefore omit the proofs. Of course there are many other sorts of model which would fit the framework of this paper.

We take the internal dynamics to be that of the Fitzhugh-Nagumo model,  the connection matrix to be random and Gibbsian, with the interaction terms driven by the firing rates of the pre and post synaptic neurons, and the connections evolving according to a learning rule. We take $d \in \lbrace 1,2,3\rbrace$. For $j\in V_n$, the governing equations are
\begin{align}
dv^j_t =&  dW^{n,j}_t\label{eq: fitzhugh 1}+\bigg(\sum_{k\in V_n}G_t^{k}(J^{n,j,k},v^j,v^{(j+k)\modd V_n})\times \\& f(v^j_t) f\big(v^{(j+k)\modd V_n}_t\big)+ v^j_t - \frac{1}{3}(v^j_t)^3 - w^j_t \bigg)dt , \nonumber \\
dw^j_t =& \big(v^j_t + \mathfrak{a} - \mathfrak{c}w^j_t\big)dt.\label{eq: fitzhugh 2}
\end{align}
We take $w^j_0 = v^j_0 = 0$ as initial conditions. Here $\mathfrak{a}$ and $\mathfrak{c}$ are positive constants, and $f$ is bounded and Lipschitz. The internal dynamics of the above equation is that of the famous Fitzhugh-Nagumo model \cite{fitzhugh:55,nagumo-arimoto-etal:62,fitzhugh:66,fitzhugh:69}.  This model distils the essential mathematical features of the Hodgkin-Huxley model, yielding excitation and transmission properties from the analysis of the biophysics of sodium and potassium flows. The variable $v$ is the `fast' variable which corresponds approximately to the voltage, and $w$ is the `slow' recovery variable which is dominant after the generation of an action potential.


We take the independent inputs to be uncorrelated Ornstein-Uhlenbeck Processes $\lbrace W^{n,j}\rbrace_{j\in V_n}$. It has been proved in \cite[Lemma 17]{faugeras-maclaurin:16} that Assumption \ref{Assumption Noise LDP} is satisfied.

\subsection{Model of Interactions}
The interaction term is a simplification of the chemical synapse models in \cite{destexhe-mainen-etal:94,ermentrout-terman:10}. It is assumed that the existence of a connection between neurons $j$ and $k$ is random and determined by the random variable $J^{n,j,k} \in \lbrace 0,1\rbrace$. The model of $J^{n,j,k}$ is outlined in the following section. In the formalism of the previous section, $\mathcal{E} = \lbrace 0,1\rbrace$, with $0 = 0$. Letting $\bar{J}$ be the maximal connection strength (so that $C_J = \bar{J}$ in Assumption \ref{Assumption NonUniform}), we define $d_{\mathcal{E}}(0,1) = \bar{J}$ and $d_{\mathcal{E}}(0,0) = d_{\mathcal{E}}(1,1) = 0$. We specify that
\[
G_t^k(0,\cdot,\cdot) = 0.
\]
Otherwise, we take $G^{k}_s(1,\cdot,\cdot)$ to evolve according to the following learning rule. We  use the following classical Hebbian Learning model (refer to \cite{gerstner-kistler:02b} for a more detailed description, and in particular equation 10.6) to specify $G^{k}_t$.

The `activity' of  neuron $j$ at time $t$ is given as $\mathfrak{v}(U^j_t)$. Here $\mathfrak{v}: \R\to \R$ is Lipschitz continuous, positive and bounded. The evolution equation is defined to be, for $X,Y \in \T$,
\begin{multline}
\frac{d}{dt}G^k_t(1,X,Y) = J^{corr}\left(\bar{J} - G^k_t(1,X,Y)\right)\mathfrak{v}(X_t)\mathfrak{v}(Y_t) - J^{dec}G^{k}_t(1,X,Y).\label{eq:Lambdadefn}
\end{multline}
Here $J^{corr},J^{dec},\bar{J}$ are non-negative constants (if we let them be zero then we obtain weights which are constant in time). Initially, if $J^{n,j,k}\neq 0$, we stipulate that
\begin{equation}
G^{k}_0(1,X,Y) := G_{ini}
\end{equation}
where $G_{ini} \in [0,\bar{J}]$ is a constant stipulating the initial strength of the weights. It is straightforward to show that there is a unique solution to the above differential equation for all $X,Y \in \C([0,T],\R^2)$. One may show that $G^{k}_t(1,\cdot,\cdot) \leq \bar{J}$. In effect, the solution defines $G^{k}_t(1,\cdot,\cdot)$ as a function $\C([0,t],\R) \times \C([0,t],\R) \to \R$, which can be shown to be uniformly Lipschitz in both of its variables, where $\C([0,t],\R)$ is endowed with the supremum norm.

The random connections $J^{n,j,k}$ are assumed to be identically zero once $|k-j|$ (modulo the torus) exceeds some threshold. 

\section{Proof of Theorem \ref{Theorem Main LDP}}\label{Section Proofs}

We  obtain the LDP for $(\Pi^n)_{n\in\Z^+}$ by applying a series of transformations to the LDP for $\big(\Pi^n_{\mcY}\big)_{n\in\Z^+}$ (which is assumed in Assumption \ref{Assumption LDP}). The proof of our main result in Theorem \ref{Theorem Main LDP} is outlined just below. The rest of the document consists of lemmas auxiliary to this proof. 

We must first define a metric on the space $\mathcal{P}\big(\bar{\T}^{Z^{d}}\big)$, which, as noted in the previous section, is the space in which the double layer empirical measure $\hat{\mu}^n(\mathcal{Y}^n)$ lives. Recalling that $\bar{V}_q := \left\lbrace (j,k) \in \Z^{2d} | j,k\in V_q\right\rbrace$ and $\bar{\T}^{\bar{V}_q} := \lbrace (Y^j,\omega^{j,k})_{(j,k)\in \bar{V}_q} | Y^j \in \T \text{ and }\omega^{j,k}\in\mathcal{E}\rbrace$, let $\bar{d}_q^{\mathcal{P}}$ be the Levy-Prokhorov metric on $\mathcal{P}\big(\bar{\T}^{\bar{V}_q}\big)$ generated by the following metric on $\bar{\T}^{\bar{V}_q}$,  
 \begin{equation}
 \bar{d}_q(\mathcal{Y},\mathcal{Z}) := \sum_{j\in V_q}\norm{R^j - X^j}_T + \bigg(\sum_{j,k\in V_q}d_{\mathcal{E}}(\omega^{j,k},\beta^{j,k})^2\bigg)^{\frac{1}{2}},
 \end{equation}
where $\mathcal{Y} := (R^j,\omega^{j,k})_{j,k\in V_q}, \mathcal{Z} := (X^j,\beta^{j,k})_{j,k\in V_q}$. Let $\pi^{\mathcal{P}}_{\bar{V}_q}: \mathcal{P}(\bar{\T}^{\Z^d}) \to \mathcal{P}(\bar{\T}^{\bar{V}_q})$ be the projection onto the marginal measure over $\bar{\T}^{\bar{V}_q}$ and define $\bar{d}^{\mathcal{P}}$ be the following complete metric on $\mathcal{P}\big(\bar{\T}^{\Z^d}\big)$ (which metrizes weak convergence with respect to the cylindrical topology),
 \begin{equation}\label{eq: metric definition 3}
 \bar{d}^{\mathcal{P}}(\mu,\nu) = \sum_{j=1}^{\infty}\min\bigg(2^{-j},\bar{d}_j^{\mathcal{P}}\big(\pi^{\mathcal{P}}_{\bar{V}_j}\mu,\pi^{\mathcal{P}}_{\bar{V}_j}\nu\big)\bigg).
 \end{equation}
\begin{proof}[Proof of Theorem \ref{Theorem Main LDP}]

For $c\in \R^+$, recalling the definition of $\mathcal{P}_{\bar{S}}\big(\bar{\T}^{\Z^d}\big)$ in \eqref{eq bar S stationary}, let
\begin{multline}\label{defn Ac}
\mathcal{A}_c  = \left\lbrace\mu \in \mathcal{P}_{\bar{S}}\big( \bar{\T}^{\Z^d}\big) | \Exp^{\mu} \left[ \norm{X^0}^2_T\right] \leq c \text{ and for all }m\geq \mathfrak{m}_0 \right. \\  \Exp^{\mu}\big[\big(\sum_{k\notin V_m}\snorm{\omega^{0,k}} \big)^2\big]  \leq c\exp\big(-(4+2\rho)TC_J|V_m| \big)|V_m|^{-2\rho-2}\text{ and }\\
\Exp^{\mu}\big[ \sum_{k\notin V_m}\snorm{\omega^{0,k}} \big] \leq c\exp\big(-(3+\rho)TC_J|V_m| \big)|V_m|^{-\rho-1},
\bigg\rbrace,
\end{multline}
where we recall that $\mathfrak{m}_0$ and $\rho$ are defined in Assumption \ref{Assumption Connection Growth}. It may be observed that $\mathcal{A}_c$ is a closed subset of $ \mathcal{P}_{\bar{S}}\big( \bar{\T}^{\Z^d}\big)$.

We now define maps $\Psi,\Psi^m:\bar{\T}^{\Z^{d}}\to \T^{\Z^d}$. These are used to transform the LDP for $(\Pi^n_{\mathcal{Y}})_{n\in\Z^+}$ into an LDP for $(\Pi^n)_{n\in\Z^+}$. Using the above definitions, for $\mathcal{Z}\in \bar{\T}^{\Z^{d}}$, with $\mathcal{Z}^{j}:= (R^j,(\omega^{j,l})_{l\in\Z^d})$, we define $\Psi(\mathcal{Z}) := X$ to be such that for any $t\in [0,T]$ and $j\in \Z^d$,
\begin{align}
X^{j}_t := \int_0^t \bigg(\mathfrak{b}(X_s^{j}) + \sum_{k\in \Z^d}\Lambda^{k}_s\big(\omega^{j,k},X^j,X^{j+k}\big)\bigg)ds+ R^j_{t}.\label{eq:fundamentalmult 1}
\end{align}
Define $\Psi^m:\bar{\T}^{\Z^{d}} \to \T^{\Z^d}$ to be $\Psi^m(\mathcal{Z}) := Z$, where $Z$ satisfies
\begin{align}\label{eq: first Psi m definition}
Z^{j}_t := \int_0^t \bigg(\mathfrak{b}(Z_s^{j}) + \sum_{k\in V_m}\Lambda^{k}_s\big(\omega^{j,k},Z^j,Z^{j+k}\big)\bigg)ds+ R^j_{t}.
\end{align}
In fact both $\Psi$ and $\Psi^m$ are only well-defined on a subset of $\bar{\T}^{\Z^d}$ (the domain of $\Psi^m$ is explained at the start of Section \ref{Section Psi m}, and the domain of $\Psi$ may be inferred from Lemma \ref{Lem Uniform Bound mu Ac}). Let 
\begin{equation}
\tilde{\Psi}^m : \cup_{c\in \R^+}\mathcal{A}_c \to \mathcal{P}_S\big(\T^{\Z^d}\big)
\end{equation}
be $\tilde{\Psi}^m(\mu) = \mu \circ (\Psi^m)^{-1}$. In Section \ref{Section Psi m} we prove that $\tilde{\Psi}^m$ is well-defined and continuous on $\mathcal{A}_c$. It follows from Lemma \ref{Lemma Solution Equivalence} that $\tilde{\Psi}^n\big(\hat{\mu}(\mathcal{Y}^n)\big) = \hat{\mu}^n(U)$. 

Now, there exists a continuous function $\grave{\Psi}^m: \mathcal{P}_{\bar{S}}\big(\bar{\T}^{\Z^d}\big) \to \mathcal{P}_S\big(\T^{\Z^d} \big)$ such that $\grave{\Psi}^m(\mu) = \tilde{\Psi}^m(\mu)$ for all $\mu \in \mathcal{A}_m$. To see this, note that we may identify $\mathcal{P}_S\big(\T^{\Z^d} \big)$ as a closed convex subset of the locally-convex vector space $\mathcal{M}$ of all finite Borel measures on $\T^{\Z^d}$ (where $\mathcal{M}$ has the topology of weak convergence). Since $\mathcal{A}_m$ is closed, the existence of $\grave{\Psi}^m$ follows from \cite[Theorem 4.1]{dugundji1951extension}. 

We define $\grave{\Psi}:\mathcal{P}_{\bar{S}}(\bar{\T}^{\Z^d}) \to \mathcal{P}_S(\T^{\Z^d})$ as follows.
\begin{align}
\grave{\Psi}(\mu) &= \mu \circ \Psi^{-1} \text{ if }\mu\in \cup_{c\geq 0}\mathcal{A}_c \label{eq: grave Psi 1}\\
\grave{\Psi}(\mu) &= \beta \text{ otherwise, for some arbitrary fixed $\beta \in \mathcal{P}_S(\T^{\Z^d})$.}\label{eq: grave Psi 2}
\end{align}
Note that it does not matter how we define $\grave{\Psi}$ outside of $\cup_{c\geq 0}\mathcal{A}_c$.

We use \cite[Theorem 4.2.23]{dembo-zeitouni:97} to prove the result. First note that by Assumption \ref{Assumption LDP}, the laws $\big(\Pi^n_{\mathcal{Y}}\big)_{n\in\Z^+}$ satisfy an LDP with good rate function. To use \cite[Theorem 4.2.23]{dembo-zeitouni:97}, since $\tilde{\Psi}^n\big(\hat{\mu}(\mathcal{Y}^n)\big) = \hat{\mu}^n(U)$ (as noted in Lemma \ref{Lemma Solution Equivalence}), we must first verify the `exponentially good approximations' property, namely that for arbitrary $\alpha,\delta > 0$, there exists an $m$ such that
\begin{equation}\label{eq: to prove 1}
\lsup{n}\frac{1}{|V_n|}\log \mathbb{P}\bigg( d^{\mathcal{P}}\big(\grave{\Psi}^m(\hat{\mu}^n(\mathcal{Y}^n)),\tilde{\Psi}^n(\hat{\mu}^n(\mathcal{Y}^n)) \big) > \delta \bigg) \leq -\alpha.
\end{equation}
To see this, by Lemma \ref{Lem Bound hat mu in Ac}, there exists a $c \in \R^+$ such that
\begin{equation}
\lsup{n}\frac{1}{|V_n|}\log \mathbb{P}\bigg( \hat{\mu}^n(\mathcal{Y}^n) \notin \mathcal{A}_c \bigg) \leq -\alpha.
\end{equation}
But by Lemma \ref{Lem Uniform Bound mu Ac}, there exists an $m \geq c$ such that if $\hat{\mu}^n(\mathcal{Y}^n) \in \mathcal{A}_c$, then \newline$d^{\mathcal{P}}\big(\grave{\Psi}^m(\hat{\mu}^n(\mathcal{Y}^n)),\tilde{\Psi}^n(\hat{\mu}^n(\mathcal{Y}^n)) \big) \leq \delta$. This establishes \eqref{eq: to prove 1}.

To apply the result in \cite[Theorem 4.2.23]{dembo-zeitouni:97}, we also require that for each $\alpha\in\R^+$, writing $\mathcal{A}_\alpha^* = \lbrace \mu \in \mathcal{P}_{\bar{S}}\big(\bar{\T}^{\Z^d} \big) | I_{\mathcal{Y}}(\mu) \leq \alpha \rbrace$,
\begin{equation}\label{eq reall last jespere}
\lim_{m\to\infty}\sup_{\mu\in \mathcal{A}_{\alpha}^*}d^{\mathcal{P}}\big(\grave{\Psi}(\mu),\grave{\Psi}^m(\mu)\big) = 0.
\end{equation}
Thanks to Lemma \ref{Lem alpha c relation}, it suffices to prove that for any $c\in \R^+$,
\begin{equation}
\lim_{m\to\infty}\sup_{\mu\in \mathcal{A}_{c}}d^{\mathcal{P}}\big(\grave{\Psi}(\mu),\grave{\Psi}^m(\mu)\big) = 0.\label{eq: last to prove jespere}
\end{equation}
Now once $m \geq c$, $\grave{\Psi}^m(\mu) = \mu \circ (\Psi^m)^{-1}$, and since $\grave{\Psi}(\mu) = \lim_{m\to\infty}\grave{\Psi}^m(\mu)$ (thanks to Lemma \ref{Lem Uniform Bound mu Ac}),
\begin{equation*}
\lim_{m\to\infty}\sup_{\mu\in \mathcal{A}_{c}}d^{\mathcal{P}}\big(\grave{\Psi}(\mu),\grave{\Psi}^m(\mu)\big) \leq 2\lim_{m\to\infty}\sup_{n\geq m}\sup_{\mu\in\mathcal{A}_c}d^{\mathcal{P}}\big(\tilde{\Psi}^m(\mu) ,\tilde{\Psi}^n(\mu) \big).
\end{equation*}
\eqref{eq: last to prove jespere} follows as a consequence of the above and Lemma \ref{Lem Uniform Bound mu Ac}.

The theorem now follows from the fact that the laws $\big(\Pi^n_{\mathcal{Y}}\big)_{n\in\Z^+}$ satisfy an LDP with good rate function, $\grave{\Psi}^m$ is continuous, \eqref{eq: to prove 1}, \eqref{eq reall last jespere} and  \cite[Theorem 4.2.23]{dembo-zeitouni:97}.
\end{proof}The following lemma notes that $\Psi^n$ preserves $\bar{V}_n$-periodicity.
\begin{lemma}\label{Lemma Solution Equivalence}
For each $W^n \in \T^{V_n}$, there exists a unique solution $U$ to \eqref{eq:fundamentalmult}. This solution satisfies, for all $p\in\Z^d$,
\begin{align*}
\Psi^n(\tilde{\mathcal{Y}}^n)^p &= U^{p \modd V_n}.
\end{align*}
Furthermore
\begin{equation}
\hat{\mu}^n(U) = \hat{\mu}^n(\mathcal{Y}^n) \circ (\Psi^n)^{-1}.\label{eq:equivalence of empirical measures}
\end{equation}
\end{lemma}
\begin{proof}
The existence and uniqueness of the solution \eqref{eq:fundamentalmult} follows from Lemma \ref{lem: Psin Psi}. It follows from the definition that for all $j\in \Z^d$, $\Psi^n(\bar{S}^j \tilde{\mathcal{Y}}^n) = S^j\Psi(\tilde{\mathcal{Y}}^n)$. \eqref{eq:equivalence of empirical measures} follows directly from this and the definition of the empirical measure.
\end{proof}
\begin{proof}[Proof of Corollary \ref{Cor: Convergence Averages}]
The assumptions mean that $\mu_{\mathcal{Y}}$ is the unique probability measure such that $I_{\mathcal{Y}}(\mu_{\mathcal{Y}}) = 0$. This means, in turn, that $\grave{\Psi}(\mu_{\mathcal{Y}})$ is the unique zero of $I$. It can then be shown using the Borel-Cantelli Lemma that $\hat{\mu}^n(U) \to \grave{\Psi}(\mu_{\mathcal{Y}})$ almost surely (see for instance \cite[Theorem II.6.4]{ellis:85}).
\end{proof}
\section{Existence and continuity of $\tilde{\Psi}^m$}\label{Section Psi m}

The map $\Psi^m$ (as defined in \eqref{eq: first Psi m definition}) is not well-defined on all of $\bar{\T}^{\Z^d}$, and it is not necessarily continuous on closed subsets either. Our first task therefore is to define $\Psi^m$ on a weighted subspace $\bar{\T}_{\lambda_m}^{\Z^d}$ of $\bar{\T}^{\Z^d}$ (which we outline in further detail below). We then show that the induced map $\mu \to \mu \circ (\Psi^m)^{-1}$ is continuous on the closed sets $\mathcal{A}_c \subset \mathcal{P}_{\bar{S}}\big(\bar{\T}^{\Z^d} \big)$. The integer $m$ is fixed throughout this section.

We start by defining the weights $\lbrace \lambda^j_m\rbrace_{j\in\Z^d}$ which are used to modulate the convergence (similar methods have been used previously in, for example,  \cite{shiga-shimizu:80}). For $\theta \in [-\pi,\pi]^d$, let $\kappa_m^k = \mathbf{1}_{k\in V_m}$, $\tilde{\kappa}_m(\theta) = \sum_{k\in \Z^d}\kappa_m^k\exp\big(-i\langle \theta,k\rangle \big)$ and
\[
\tilde{\lambda}_m(\theta) = h\bigg(\rho|V_m| - \tilde{\kappa}_m(\theta) \bigg)^{-1},
\]
where we recall that $\rho>1$ is used in the definition of $\mathcal{A}_c$ in \eqref{defn Ac} and $h$ is chosen to ensure that
\begin{equation}\label{defn h}
\frac{1}{(2\pi)^d}\int_{[-\pi,\pi]^d}\tilde{\lambda}_m(\theta)d\theta = 1.
\end{equation}
Let $\lambda_m^j= \frac{1}{(2\pi)^d}\int_{[-\pi,\pi]^d}\exp\big( i\langle\theta,j\rangle\big)\tilde{\lambda}_m(\theta)d\theta$. In the case that $\rho =2$, it was proved in \cite[Lemma 5]{faugeras-maclaurin:16} that 
\begin{align}\label{eq total lambda}
\sum_{j\in\Z^d}\lambda^j_m &= 1, \\
\lambda^j_m &> 0,\\
\sum_{k\in V_m}\lambda_m^{j-k} \kappa^k_m &\leq \rho|V_m|\lambda^j_m. \label{eq: lambdakappa2prelim}
\end{align}
The above identities easily generalise to any $\rho > 1$. An immediate consequence of \eqref{eq: lambdakappa2prelim} is, for any non-negative $(\upsilon^j)_{j\in\Z^d} \subset \R$,
\begin{equation}
\sum_{j\in\Z^d,k\in V_m}\lambda^j_m\upsilon^{j+k} \leq \rho|V_m|\sum_{j\in \Z^d}\lambda^j_m\upsilon^j, \label{eq: lambdakappa2}
\end{equation}
as long as the above sums are finite. Let $\bar{\T}_{\lambda_m}^{\Z^d}$ be the space of all $\mathcal{Z} := \big(R^j,(\omega^{j,k})_{k\in\Z^d}\big)_{j\in\Z^d} \in \bar{\T}^{\Z^d}$ such that
\begin{align}
\sum_{j\in\Z^d}\lambda^j_m \norm{R^j}_T^2 < \infty \\ 
\sum_{j\in\Z^d}\lambda^j_m \sum_{k\in\Z^d}\snorm{\omega^{j,k}}^2 < \infty.
\end{align}

Define $\Psi^m: \bar{\T}^{\Z^d}_{\lambda_m}\to \T^{\Z^d}$ as follows. Using the above definitions, for $\mathcal{Z}\in \bar{\T}^{\Z^d}$, with $\mathcal{Z}^{j} := \big(R^j,(\omega^{j,l})_{l\in\Z^d}\big)$, we define $\Psi^m(\mathcal{Z}) := X = (X^j_t)_{j\in\Z^d,t\in [0,T]}$ to be such that for any $t\in [0,T]$ and $j\in \Z^d$,
\begin{align}
X^{j}_t &:= \int_0^t \bigg(\mathfrak{b}(X_s^{j}) + \sum_{k\in V_m}\Lambda^{k}_s\big(\omega^{j,k},X^j,X^{(j+k)\modd V_m}\big)\bigg)ds+ R^j_{t}.\label{eq:fundamentalmult 2}
\end{align}

\begin{lemma}\label{lem: Psin Psi}
For each $\mathcal{Z}\in \bar{\T}_{\lambda_m}^{\Z^{d}}$, $\Psi^m(\mathcal{Z})$ exists and is unique.
\end{lemma}
\begin{proof}
Fix $\mathcal{Z} = (R^j,\omega^{j,k})_{j,k\in\Z^d} \in \bar{\T}^{\Z^{d}}_{\lambda_m}$. We assume that $\omega^{j,k} = 0$ if $j-k \notin V_m$, because if $\Psi^m(\mathcal{Z})$ were to exist then it would be independent of these values. We prove the existence in two steps. We start by finding a series of periodic approximations $\mathcal{Z}(p)$ of $\mathcal{Z}$. $\Psi^m(\mathcal{Z}(p))$ can be shown to exist through the Cauchy-Peano existence theorem, and then it will be shown that $\lbrace \Psi^m(\mathcal{Z}(p))\rbrace_{p\in\Z^+}$ is Cauchy, yielding the lemma. 

For $p\in \Z^+$, let $\mathcal{Z}(p) \in \bar{\T}^{\Z^{d}}$ have the property that, writing $\mathcal{Z}(p)^{j} = \big(R(p)^j,(\beta(p)^{j,k})_{k\in\Z^d}\big)$ $R(p)^j = R(p)^a$, $\beta(p)^{j,k} = \beta(p)^{a,k}$, where $a = j\modd V_p$. Very similarly to \cite[Lemma 6]{faugeras-maclaurin:16}, we can choose $\lbrace \mathcal{Z}(p)\rbrace_{p\in\Z^+}$ such that
\begin{align}
\lim_{p\to \infty}\sum_{j\in\Z^d}\lambda_m^j \norm{R(p)^j - R^j}_T = 0\label{eq assmp 1 Cauchy}\\
\lim_{p\to\infty}\sum_{j \in\Z^d, k\in V_m}\lambda_m^j d_{\mathcal{E}}(\beta(p)^{j,k},\beta^{j,k})^2 = 0.\label{eq assmp 3 Cauchy}
\end{align}

Now $\Psi^m(\mathcal{Z}(p))$ must exist for any $p\in\Z^+$. In brief, this is because for any $r\in \Z^d$, if the solution were to exist then $S^r\Psi^m(\mathcal{Z}(p)) = \Psi^m(\bar{S}^r\mathcal{Z}(p))$. We can therefore reduce the existence of $\Psi^m(\mathcal{Z}(p))$ to the existence of a solution to a finite-dimensional ordinary differential equation indexed over $V_p$, as in \cite[Lemma 7]{faugeras-maclaurin:16}. The solution $\Psi^m(\mathcal{Z}(p))$ exists for any $p\in\Z^+$ thanks to the generalised form of the Cauchy-Peano existence theorem in \cite[Lemma 8]{faugeras-maclaurin:16}. Now, by Lemma \ref{Lemma: Bound XZ 1},
\begin{multline*}
\sum_{j\in\Z^d}\lambda_m^j\norm{\Psi^m(\mathcal{Z}(q))^j - \Psi^m(\mathcal{Z}(p))^j}_T \leq \exp\big(TC+(1+\rho)TC_J|V_m| \big)\times \\ \bigg[  2\sum_{j\in\Z^d}\lambda_m^j\norm{R(p)^j-R(q)^j}_T +T(1+\sqrt{\rho})|V_m|^{\frac{1}{2}}\big(\sum_{j\in\Z^d}\lambda_m^j\norm{\Psi^m(\mathcal{Z}(p))^j}_T^2\big)^{\frac{1}{2}} \times \\ \bigg(\sum_{j \in\Z^d, k\in V_m}\lambda_m^j d_{\mathcal{E}}(\beta(p)^{j,k},\beta(q)^{j,k})^2\bigg)^{\frac{1}{2}}\bigg].
\end{multline*}
By Lemma \ref{Lemma: bound exp Z}, since $(A+B)^2 \leq 2A^2 + 2B^2$ and $\snorm{\cdot} \leq C_J$,
\begin{equation*}
\norm{\Psi^m(\mathcal{Z}(p))^j}_T^2 \leq 2\big[\exp\big(2T(C+2C_J |V_m|)\big) + 4\exp\big(2T(C+C_J |V_m|)\big)\norm{R(p)^j}_T^2\big].
\end{equation*}
This means that
\begin{align*}
\sum_{j\in\Z^d}\lambda_m^j \norm{\Psi^m(\mathcal{Z}(p))^j}_T^2 \leq 2\exp\big(2T\big(C+ 2C_J|V_m|\big)\big)\big( 1 + 4\sum_{j\in Z^d}\lambda_m^j \norm{R(p)^j}^2_T\big).
\end{align*}
Since
\begin{align*}
\lim_{p\to \infty}\sum_{j\in Z^d}\lambda_m^j \norm{R(p)^j}^2_T = \sum_{j\in Z^d}\lambda_m^j \norm{R^{j}}^2_T, 
\end{align*}
there must exist a uniform upper bound $L$ such that \newline$\sum_{j\in\Z^d}\lambda_m^j \norm{\Psi^m(\mathcal{Z}(p))^j}_T^2 \leq L$ for all $p\in\Z^+$. 
It is therefore a consequence of \eqref{eq assmp 1 Cauchy}-\eqref{eq assmp 3 Cauchy} that
\[
\lim_{q\to\infty}\sup_{p \geq q}\sum_{j\in\Z^d}\lambda_m^j\norm{\Psi^m(\mathcal{Z}(q))^j - \Psi^m(\mathcal{Z}(p))^j}_T = 0.
\]
Let
\[
\Psi^m(\mathcal{Z})^j = \lim_{p\to \infty}\Psi^m(\mathcal{Z}(p))^j.
\]
To show that $\Psi^m(\mathcal{Z})$ satisfies \eqref{eq:fundamentalmult 2}, we must verify that
\begin{align*}
\int_0^t \sum_{k\in V_m}\Lambda^k_s\big(\beta(p)^{j,k},\Psi^m(\mathcal{Z}(p))^j,\Psi^m(\mathcal{Z}(p))^{j+k}  \big) ds &\to \int_0^t \sum_{k\in V_m}\Lambda^k_s\big(\beta^{j,k},\Psi^m(\mathcal{Z})^j,\Psi^m(\mathcal{Z})^{j+k}\big)ds\\
\int_0^t \mathfrak{b}\big(\Psi_s^m(\mathcal{Z}(p))^j \big) ds &\to \int_0^t \mathfrak{b}\big(\Psi_s^m(\mathcal{Z})^j\big) ds.
\end{align*}
In fact the above identities follow from Assumptions \ref{Assumption Absolute Bound} and \ref{Assumption NonUniform} and the dominated convergence theorem. The uniqueness of $\Psi^m(\mathcal{Z})$ follows from Lemma \ref{Lemma: Bound XZ 1}.
\end{proof}

Recall that $\tilde{\Psi}^m(\mu) := \mu \circ (\Psi^m)^{-1}$. The following lemma notes that this is well-defined for $\mu \in \mathcal{A}_c$ ($\mathcal{A}_c$ is defined in \eqref{defn Ac}).
\begin{lemma}\label{Lem Prokhorov Technical One}
$\tilde{\Psi}^m:\mathcal{A}_c \to \mathcal{P}_S\big(\T^{\Z^d}\big)$ is well-defined and uniformly continuous. That is, for all $\epsilon$ there exists a $\delta$ such that for all $\mu,\nu \in \mathcal{A}_c$ satisfying $\bar{d}^{\mathcal{P}}(\mu,\nu) \leq \delta$,
$d^{\mathcal{P}}(\tilde{\Psi}^m(\mu),\tilde{\Psi}^m(\nu)) \leq \epsilon$. Furthermore, 
\begin{equation}\label{eq bound norm squared 1}
\Exp^{\mu}\left[ \norm{\Psi^m(\mathcal{Z})^j}_T^2 \right] < \infty,
\end{equation}
the above expectation being the same for all $j\in\Z^d$ since $\tilde{\Psi}^m(\mu)$ is stationary.
\end{lemma}
\begin{proof}
We first observe that if $\mu\in\mathcal{A}_c$, and $\mu$ is the law of $\lbrace \mathcal{Z}^j\rbrace_{j\in\Z^d}$, where $\mathcal{Z}^j = \big(R^j,(\omega^{j,k})_{k\in\Z^d}\big)$, then since 
\[
 \Exp^{\mu} \left[ \sum_{j\in\Z^d}\lambda^j_m\norm{R^j}^2_T\right] \leq c\sum_{j\in\Z^d}\lambda^j_m < \infty,
\]
$\sum_{j\in\Z^d}\lambda^j_m\norm{R^j}^2_T < \infty$, $\mu$-almost surely. Similarly, since $\Exp^{\mu} \bigg[ \sum_{k\in\Z^d}\snorm{\omega^{j,k}}^2 \bigg] < \infty$ (by the stationarity of $\mu$, this integral is independent of $j$), it must be that $ \sum_{j\in\Z^d,k\in\Z^d}\lambda^j_m\snorm{\omega^{j,k}}^2 < \infty$, $\mu$-almost surely. Thanks to Lemma \ref{lem: Psin Psi}, this means that $\Psi^m(\mathcal{Z})$ exists for $\mu$ almost-every $\mathcal{Z}$, and therefore that $\tilde{\Psi}^m(\mu)$ is well-defined. 

The bound in Lemma \ref{Lemma: bound exp Z} implies \eqref{eq bound norm squared 1}, since $\Exp\left[\norm{R^j}_T^2 \right] \leq c$.
 
Now it is an immediate consequence of the definition of the metric in \eqref{defn prohorov metric T} that for any $s\in \Z^+$,
\begin{multline}
\lim_{\gamma\to 0}\sup_{\mu,\nu \in \mathcal{A}_c, \bar{d}^{\mathcal{P}}(\mu,\nu) \leq \gamma}d^{\mathcal{P}}\bigg( \nu\circ(\Psi^m)^{-1},\mu\circ(\Psi^m)^{-1}\bigg) 
\leq 2^{-s-1} + \\
\sum_{r=1}^s\overline{\lim_{\gamma\to 0}}\sup_{\mu,\nu \in \mathcal{A}_c, \bar{d}^{\mathcal{P}}(\mu,\nu) \leq \gamma}d^{\mathcal{P}}_r\bigg( \pi^{\mathcal{P}}_{V_r}\big(\nu\circ(\Psi^m)^{-1}\big),\pi^{\mathcal{P}}_{V_r}\big(\mu\circ(\Psi^m)^{-1}\big)\bigg).
\end{multline}
It thus suffices for us to prove that for an arbitrary $r\in\Z^+$,
\begin{equation}\label{eq suffice 1}
\lim_{\gamma\to 0}\sup_{\mu,\nu \in \mathcal{A}_c, \bar{d}^{\mathcal{P}}(\mu,\nu) \leq \gamma}d^{\mathcal{P}}_r\bigg( \pi^{\mathcal{P}}_{V_r}\big(\nu\circ(\Psi^m)^{-1}\big),\pi^{\mathcal{P}}_{V_r}\big(\mu\circ(\Psi^m)^{-1}\big)\bigg) = 0.
\end{equation}
Let $\mathcal{B}^{V_r}$ be the set of all $V_r$ cylinder sets, that is every $B \in \mathcal{B}^{V_r} \subset \mathcal{B}\big(\T^{\Z^d}\big)$ is such that $B = (\pi^{V_r})^{-1}B_2$ for some $B_2 \in \mathcal{B}\big(\T^{V_r}\big)$. For $\delta > 0$, let $B^\delta \subset B$ be the $\delta$-interior relative to the norm on $\T^{V_r}$, i.e. such that
\begin{equation}
B^\delta = \left\lbrace X \in \T^{\Z^d} | \bigg\lbrace Y\in\T^{\Z^d} | \sum_{j\in V_r}\norm{X^j-Y^j}_T \leq \delta\bigg\rbrace \subseteq B \right\rbrace.
\end{equation}
Let $\bar{B} = (\Psi^m)^{-1}(B)$ and  $\bar{B}^\delta = (\Psi^m)^{-1}(B^\delta)$. We use Lemma \ref{Lemma Intermediate Technical} to establish \eqref{eq suffice 1}. This lemma implies that for any $\delta$ there exists a $\gamma$ such that
\begin{equation}
\sup_{\mu,\nu \in \mathcal{A}_c, \bar{d}^{\mathcal{P}}(\mu,\nu) \leq \gamma}\sup_{B \in \mathcal{B}^{V_r}}\big(\nu(\bar{B}) - \mu(\bar{B}^{\delta}\big)\big)\leq \delta.
\end{equation}
It follows from the above equation and the definition of the Prokhorov metric that 
\begin{equation}
\sup_{\mu,\nu \in \mathcal{A}_c, \bar{d}^{\mathcal{P}}(\mu,\nu) \leq \gamma}d^{\mathcal{P}}_r\bigg( \pi^{\mathcal{P}}_{V_r}\big(\nu\circ(\Psi^m)^{-1}\big),\pi^{\mathcal{P}}_{V_r}\big(\mu\circ(\Psi^m)^{-1}\big)\bigg) \leq \delta.
\end{equation}
Since $\delta$ is arbitrary, \eqref{eq suffice 1} is an immediate consequence of this, which completes the proof.
\end{proof}
The variables in the following lemma are defined in the proof of Lemma \ref{Lem Prokhorov Technical One}.
\begin{lemma}\label{Lemma Intermediate Technical}
For any $r\in\Z^+$ and $\delta > 0$,
\begin{equation}
\overline{\lim_{\gamma \to 0}}\sup_{\mu,\nu \in \mathcal{A}_c, \bar{d}^{\mathcal{P}}(\mu,\nu) \leq \gamma}\sup_{B \in \mathcal{B}^{V_r}}\big(\nu(\bar{B}) - \mu(\bar{B}^\delta\big))\leq \delta.
\end{equation}
\end{lemma}
\begin{proof}
For $\epsilon > 0$ and $q > m$, let $\mathcal{H}^\epsilon_q \subset \bar{\T}_{\lambda_m}^{\Z^d}$ be the subset of all $\mathcal{Z} = (R^j,\beta^{j,k})_{j,k\in\Z^d}$ such that, writing $Z = \Psi^m(\mathcal{Z})$,
\begin{align}
\sum_{j\notin V_q}\lambda^j_m\norm{R^j}_T &\leq \epsilon, \label{eq: defn H 1}\\
\sum_{j\in\Z^d}\lambda^j_m\norm{Z^j}^2_T &\leq \epsilon^{-1},\label{eq: defn H 2} \\
\sum_{j\notin V_q,k\in V_m}\lambda^j_m \snorm{\beta^{j,k}}^2&\leq \epsilon^2.\label{eq: defn H 3}
\end{align}
Writing $\bar{B}^\epsilon_q = \bar{B}\cap\mathcal{H}^\epsilon_q$ and $\bar{B}^{\delta,\epsilon}_q = \bar{B}^\delta\cap\mathcal{H}^\epsilon_q$, it may be observed that 
\begin{equation}\label{eq split mu nu}
\nu\big(\bar{B} \big) - \mu\big( \bar{B}^\delta\big) \leq \nu\big(\bar{B}^\epsilon_q \big) -\mu\big( \bar{B}^{\delta,\epsilon}_q\big) + \nu\big((\mathcal{H}^\epsilon_q)^c \big).
\end{equation}
Suppose that $\mathcal{X},\mathcal{Z} \in \mathcal{H}^\epsilon_q$, with  $\mathcal{X}:= \big(Q^j,\omega^{j,k}\big)_{j,k\in\Z^d}$, $\mathcal{Z} := \big(R^j,\beta^{j,k}\big)_{j,k\in\Z^d} \in \bar{\T}_\lambda^{\Z^d}$. Let $X = \Psi^m\big(\mathcal{X}\big)$ and $Z = \Psi^n\big(\mathcal{Z}\big)$. By Lemma \ref{Lemma: Bound XZ 1}, for any $r\in \Z^+$,
\begin{equation*}
\sum_{j\in V_r}\lambda^j_m\norm{X^j - Z^j}_T \leq A_m\bigg(H^q_1 + H^q_2\bigg) ,
\end{equation*}
where $A_m = T\exp\big(TC+T(1+\rho)C_J|V_m| \big)$,
\begin{multline*}
H^q_1 =2 \sum_{j\in V_q}\lambda^j_m\norm{Q^j - R^j}_T  \\ +(1+\sqrt{\rho})|V_m|^{\frac{1}{2}}\big(\sum_{j\in\Z^d}\lambda^j_m\norm{Z^j}^2_s\big)^{\frac{1}{2}} \bigg(\sum_{j \in V_q, k\in V_m}\lambda^j_m d_{\mathcal{E}}(\omega^{j,k},\beta^{j,k})^2\bigg)^{\frac{1}{2}}
\end{multline*}
and
\begin{align*}
H^q_2 &=  2\sum_{j\notin V_q}\lambda^j_m \norm{Q^j - R^j}_T\\ &+ (1+\sqrt{\rho})|V_m|^{\frac{1}{2}}\big(\sum_{j\in\Z^d}\lambda^j_m\norm{Z^j}^2_s\big)^{\frac{1}{2}} \bigg(\sum_{j \notin V_q, k\in V_m}\lambda^j_m d_{\mathcal{E}}(\omega^{j,k},\beta^{j,k})^2\bigg)^{\frac{1}{2}}\\
&\leq  2\sum_{j\notin V_q}\lambda^j_m\big[\norm{Q^j}_T +\norm{R^j}_T\big] \\ &+ (1+\sqrt{\rho})\sqrt{2}|V_m|^{\frac{1}{2}}\big(\sum_{j\in\Z^d}\lambda^j_m\norm{Z^j}^2_s\big)^{\frac{1}{2}} \bigg(\sum_{j \notin V_q, k\in V_m}\lambda^j_m \big(\snorm{\omega^{j,k}}^2+ \snorm{\beta^{j,k}}^2\big)\bigg)^{\frac{1}{2}},
\end{align*}
recalling that $\snorm{\cdot} = d_{\mathcal{E}}(\cdot,0)$. We have used the inequalities $\sqrt{F+G} \leq \sqrt{F} + \sqrt{G}$, and $(F+G)^2 \leq 2F^2 + 2G^2$. Using the definition of $\mathcal{H}^\epsilon_q$, and also the Cauchy-Schwarz Inequality, we find that
\begin{align*}
H^q_2 &\leq 4\epsilon +  \sqrt{2}(1+\sqrt{\rho})|V_m|^{\frac{1}{2}}\epsilon^{-\frac{1}{2}}\epsilon \\
&= 4\epsilon +  \sqrt{2}(1+\sqrt{\rho})|V_m|^{\frac{1}{2}}\epsilon^{\frac{1}{2}}.
\end{align*}
Take $\epsilon$ to be sufficiently small that $H^q_2 \leq \frac{\delta}{2A_m}$ and 
\begin{equation}\label{eq LM2}
4c\bigg(\frac{1}{2}\epsilon^{-1}\exp\big(-2T(C+2C_J|V_m|) \big)-1 \bigg)^{-1} \leq \frac{\delta}{4}.
\end{equation}
We find that
\begin{equation}\label{eq: bound r 1}
\sum_{j\in V_r}\lambda^j_m\norm{X^j - Z^j}_T \leq A_m \bar{d}_q(\mathcal{X},\mathcal{Z}) \bigg(2 + (1+\sqrt{\rho})|V_m|^{\frac{1}{2}}\epsilon^{-\frac{1}{2}}\bigg) 
+\frac{\delta}{2}.
\end{equation}
By Lemma \ref{Lemma bound mu 4}, and noting \eqref{eq LM2}, we may take $q$ to be sufficiently large that for all $\nu \in \mathcal{A}_c$,
 \begin{equation}\label{eq q large assumption}
 \nu\bigg( \big(\mathcal{H}^\epsilon_q \big)^c\bigg) \leq \frac{\delta}{2}. 
 \end{equation}
The lemma will follow once we have established the following claim.

\textit{Claim:}

We claim that for $\gamma$ sufficiently small,
\begin{equation}\label{eq: claim gamma}
\sup_{\mu,\nu \in \mathcal{A}_c, \bar{d}^{\mathcal{P}}(\mu,\nu)\leq \gamma}\sup_{B \in \mathcal{B}_m} \nu\big(\bar{B} \big) -\mu\big( \bar{B}^\delta \big) \leq 2^q\gamma + \frac{\delta}{2}.
\end{equation}
Thanks to \eqref{eq split mu nu} and \eqref{eq q large assumption}, it suffices to prove that, for $\gamma$ sufficiently small,
\[
\sup_{\mu,\nu \in \mathcal{A}_c, \bar{d}^{\mathcal{P}}(\mu,\nu)\leq \gamma}\sup_{B \in \mathcal{B}_m} \nu\big(\bar{B}^\epsilon_q \big) -\mu\big( \bar{B}^{\delta,\epsilon}_q\big) \leq \gamma 2^q.
\]
For $\eta > 0$, let $\tilde{B}^{\eta,\epsilon}_q \subset \bar{B}^\epsilon_q$ be the $\eta$ relative interior, that is
\begin{equation}
\tilde{B}^{\eta,\epsilon}_q = \left\lbrace \mathcal{X} \in \bar{B}^\epsilon_q | \big\lbrace \mathcal{W} \in \bar{\T}^{\Z^d} | \bar{d}_q\big(\mathcal{X},\mathcal{W} \big) \leq \eta\big\rbrace \subseteq \bar{B}^{\epsilon}_q \right\rbrace.
\end{equation}
It follows from \eqref{eq: bound r 1} that if
\begin{equation}\label{eq: gamma bound 3}
\eta \leq \frac{\delta}{2}\bigg(A_m\big(2+(1+\sqrt{\rho})|V_m|^{\frac{1}{2}}\epsilon^{-\frac{1}{2}}\big) \bigg)^{-1},
\end{equation}
then
\[
\tilde{B}^{\eta,\epsilon}_q \subseteq \bar{B}^{\delta,\epsilon}_q.
\]
This means that if \eqref{eq: gamma bound 3} is satisfied, then
\begin{align*}
\nu\big(\bar{B}^\epsilon_q \big) -\mu\big( \bar{B}^{\delta,\epsilon}_q\big) &\leq \nu\big(\bar{B}^\epsilon_q \big) -\mu\big( \tilde{B}^{\eta,\epsilon}_q\big)\\
&\leq \bar{d}^{\mathcal{P}}_q(\nu,\mu)\text{ if }\bar{d}^{\mathcal{P}}_q(\nu,\mu) \leq \eta,
\end{align*}
thanks to the definition of the Prokhorov Metric. Now it follows from the definition that if $\bar{d}^{\mathcal{P}}(\mu,\nu) \leq \gamma$, then $\bar{d}^{\mathcal{P}}_q(\nu,\mu) \leq 2^q\gamma$. We have thus established our claim in \eqref{eq: claim gamma}.

\end{proof}
\begin{lemma}\label{Lemma bound mu 4}
For all $\mu \in \mathcal{A}_c$, and $q$ sufficiently large,
\begin{equation}
\mu\big( \big(\mathcal{H}^\epsilon_q \big)^c\big) \leq  \sqrt{c}\epsilon^{-1} \sum_{j\notin V_q}\lambda^j_m + c L(m,\epsilon),
\end{equation}
where $L(m,\epsilon)$ is defined in \eqref{eq: L m epsilon} and $\mathcal{H}^\epsilon_q$ is defined in \eqref{eq: defn H 1} -\eqref{eq: defn H 3}.
\end{lemma}
\begin{proof}
We suppose that $\mu$ is the law of $\mathcal{Z} = (R^j,\beta^{j,k})_{j,k\in\Z^d}$ and write $Z = \Psi^m(\mathcal{Z})$. We obtain bounds on each of \eqref{eq: defn H 1} -\eqref{eq: defn H 3}. By Chebyshev's Inequality,
\begin{align*}
\mu\bigg(\sum_{j\notin V_q}\lambda^j_m\norm{R^j}_T \geq \epsilon\bigg) \leq & \epsilon^{-1}\Exp^{\mu}\big[ \sum_{j\notin V_q}\lambda_m^j\norm{R^j}_T\big] \\
\leq & \epsilon^{-1}\sum_{j\notin V_q}\lambda^j_m\Exp^{\mu}\big[ \norm{R^j}^2_T\big]^{\frac{1}{2}} \\
\leq & \sqrt{c}\epsilon^{-1}\sum_{j\notin V_q}\lambda^j_m,
\end{align*}
using the definition of $\mathcal{A}_c$ in \eqref{defn Ac}. By Lemma \ref{Lemma: bound exp Z}, the bound $\snorm{\cdot} \leq C_J$ of Assumption \ref{Assumption NonUniform}, and the fact that $(A+B)^2 \leq 2(A^2+B^2)$, for each $j\in\Z^d$,
\[
\norm{Z^j}_T^2 \leq 2\big(1+4\norm{R^j}_T^2\big)\exp\big(2T(C+2C_J|V_m|) \big),
\]
and therefore
\begin{align*}
\norm{R^j}_T^2 &\geq \frac{1}{4}\bigg(\frac{1}{2}\norm{Z^j}_T^2\exp\big(-2T(C+2C_J|V_m|) \big)-1 \bigg)\\
\sum_{j\in\Z^d}\lambda^j_m\norm{R^j}_T^2 &\geq \frac{1}{4}\bigg(\frac{1}{2}\sum_{j\in\Z^d}\lambda^j_m\norm{Z^j}_T^2\exp\big(-2T(C+2C_J|V_m|) \big)-1 \bigg),
\end{align*}
after summing over $j$. Hence, writing 
\begin{equation}\label{eq: L m epsilon}
L(m,\epsilon) = 4\bigg(\frac{1}{2}\epsilon^{-1}\exp\big(-2T(C+2C_J|V_m|) \big)-1 \bigg)^{-1},
\end{equation}
and noting that $\sum_{j\in \Z^d}\lambda^j_m = 1$,
\begin{align*}
\mu\bigg[\sum_{j\in\Z^d}&\lambda_m^j\norm{Z^j}^2_T \geq \epsilon^{-1}\bigg] \\ &\leq \mu \bigg[ \sum_{j\in\Z^d}\lambda^j_m \norm{R^j}_T^2  \geq  \frac{1}{4}\bigg(\frac{1}{2}\epsilon^{-1}\exp\big(-2T(C+2C_J|V_m|) \big)-1 \bigg)\bigg] \\
&\leq L(m,\epsilon)\Exp^{\mu}\bigg[ \sum_{j\in\Z^d}\lambda^j_m\norm{R^j}^2_T\bigg]  \\
&\leq cL(m,\epsilon),
\end{align*}
where we have made use of Chebyshev's Inequality. Finally, observe that thanks to Assumption \ref{Assumption NonUniform},
\begin{align*}
\sum_{j\notin V_q}\lambda^j_m\sum_{k\in V_m}\snorm{\beta^{j,k}}^2 \leq \sum_{j\notin V_q}\lambda^j_m |V_m| C_J^2.
\end{align*}
Since $\sum_{j \in \Z^d}\lambda^j_m = 1$, for $q$ sufficiently large the above must be less than or equal to $\epsilon^2$, so that \eqref{eq: defn H 3} is satisfied.
\end{proof}\begin{lemma}\label{Lemma: Bound XZ 1}
Suppose that $\mathcal{X} = (Q^j,\omega^{j,k})_{j,k\in\Z^d} \in \bar{\T}^{\Z^d}_{\lambda_m}$ and $\mathcal{Z} = (R^j,\beta^{j,k})_{j,k\in\Z^d} \in \bar{\T}^{\Z^d}_{\lambda_m}\cap\bar{\T}^{\Z^d}_{\lambda_n}$. Suppose that for $n\geq m$ there exist solutions $X = \Psi^m(\mathcal{X})$ and $Z = \Psi^n(\mathcal{Z})$ to \eqref{eq:fundamentalmult 2}. Then
\begin{multline*}
\sum_{j\in\Z^d}\lambda_m^j\norm{X^j - Z^j}_T \leq \exp\big( TC+T(1+\rho)C_J |V_m| \big)\\ \times\bigg[(1+\sqrt{\rho})T|V_m|^{\frac{1}{2}}\big(\sum_{j\in\Z^d}\lambda_m^j\norm{Z^j}^2_T\big)^{\frac{1}{2}} \big(\sum_{j \in \Z^d,k\in V_m}\lambda_m^j d_{\mathcal{E}}(\omega^{j,k},\beta^{j,k})^2\big)^{\frac{1}{2}} \\+ 2 \sum_{j\in \Z^d}\lambda_m^j\norm{Q^j - R^j}_T 
+ T\sum_{j\in \Z^d,k\in V_n - V_m}\lambda_m^j\snorm{\beta^{j,k}}\big(\norm{Z^j}_T + 1\big)\bigg]. 
\end{multline*}
\end{lemma}
\begin{proof}
Observe that
\begin{align*}
X^j_t - Z^j_t = \int_0^t\bigg( \mathfrak{b}(X_s^j) - \mathfrak{b}(Z_s^j)-\sum_{k\in V_n - V_m}\Lambda^k_s(\beta^{j,k},Z^j,Z^{j+k}) \\ + \sum_{k\in V_m}\Lambda^k_s(\omega^{j,k},X^j,X^{j+k})-\Lambda^k_s(\beta^{j,k},Z^j,Z^{j+k})\bigg)ds + Q^j_t - R^j_t \\
= \int_0^t\bigg( \mathfrak{b}(X_s^j) - \mathfrak{b}(Z_s^j)-\sum_{k\in V_n - V_m}\Lambda^k_s(\beta^{j,k},Z^j,Z^{j+k}) \\ + \sum_{k\in V_m}\bigg[\Lambda^k_s(\omega^{j,k},X^j,X^{j+k})-\Lambda^k_s(\omega^{j,k},Z^j,X^{j+k})\\+\Lambda^k_s(\omega^{j,k},Z^j,X^{j+k})-\Lambda^k_s(\omega^{j,k},Z^j,Z^{j+k})\\+\Lambda^k_s(\omega^{j,k},Z^j,Z^{j+k})-\Lambda^k_s(\beta^{j,k},Z^j,Z^{j+k})\bigg]\bigg)ds + Q^j_t - R^j_t.
\end{align*}
Lets first assume that $X^j_t \geq Z^j_t$. In this case, define $[\tau,\gamma]$ to be such that $X^j_t \geq Z^j_t$ for all $t\in [\tau,\gamma]$ and $\tau$ is as small as possible.
Applying the inequalities of Assumptions \ref{Assumption Absolute Bound} and \ref{Assumption NonUniform}, we find that for $1\leq p \leq m$, and $s\in [\tau,\gamma]$,
\begin{align*}
\mathfrak{b}_p(X_s^j) - \mathfrak{b}_p(Z_s^j) \leq C | X_s^j - Z_s^j |.
\end{align*}
This means that
\begin{align*}
\sup_{1\leq p \leq m}\big| X^{p,j}_t - Z^{p,j}_t\big| \leq \sup_{1\leq p \leq m}\big| X^{p,j}_\tau - Z^{p,j}_\tau \big| + \int_\tau^t\bigg(C\norm{X^j - Z^j}_s\\+ \sum_{k\in V_n - V_m}\snorm{\beta^{j,k}}\big(\norm{Z^j}_s + 1\big) 
 C_J |V_m|\norm{X^j - Z^j}_s + C_J\sum_{k\in V_m}\norm{X^{j+k}-Z^{j+k}}_s \\+ \sum_{k\in V_m}d_{\mathcal{E}}(\omega^{j,k},\beta^{j,k})\big(\norm{Z^j}_s + \norm{Z^{j+k}}_s\big)  \bigg)ds + Q^j_t - R^j_t - Q^j_\tau + R^j_\tau.
\end{align*}
We take $\tau$ to be such that, either $\tau=0$, or $X^j_t < Z^j_t$ as $t\to \tau^-$. Now since if $\tau = 0$, then $X^j_\tau -Z^j_\tau = Q^j_0 - R^j_0$, whichever of these cases holds we find that for all $t\in [\tau,\gamma]$,
\begin{align*}
\big|X^j - Z^j\big|_t \leq  2\norm{Q^j - R^j}_T+ \int_\tau^t\bigg(C\norm{X^j - Z^j}_s + C_J |V_m|\norm{X^j - Z^j}_s \\+ C_J\sum_{k\in V_m}\norm{X^{j+k}-Z^{j+k}}_s + \sum_{k\in V_m}d_{\mathcal{E}}(\omega^{j,k},\beta^{j,k})\big(\norm{Z^j}_s + \norm{Z^{j+k}}_s\big) \\ 
+ \sum_{k\in V_n - V_m}\snorm{\beta^{j,k}}\big(\norm{Z^j}_s + 1\big)\bigg)ds\\ 
\leq  2\norm{Q^j - R^j}_T+ \int_0^t \bigg((C+ C_J |V_m|)\norm{X^j - Z^j}_s + C_J\sum_{k\in V_m}\norm{X^{j+k}-Z^{j+k}}_s \\+ \norm{Z^j}_s \sum_{k\in V_m}d_{\mathcal{E}}(\omega^{j,k},\beta^{j,k}) +  \big(\sum_{k\in V_m}\norm{Z^{j+k}}^2_s\big)^{\frac{1}{2}}\big(\sum_{k\in V_m}d_{\mathcal{E}}(\omega^{j,k},\beta^{j,k})^2\big)^{\frac{1}{2}}
\\+ \sum_{k\in V_n - V_m}\snorm{\beta^{j,k}}\big(\norm{Z^j}_s + 1\big)\bigg)ds.
\end{align*}
We could demonstrate the same inequality if $[\tau,\gamma]$ were such that $X^j_t \leq Z^j_t$ for all $t\in [\tau,\gamma]$.

We thus find that, making use of the Cauchy-Schwarz inequality,
\begin{align*}
&\sum_{j\in\Z^d}\lambda_m^j\norm{X^j - Z^j}_t \leq  2\sum_{j\in \Z^d}\lambda_m^j\norm{Q^j - R^j}_T\\&+ \int_0^t \bigg[ (C+ C_J |V_m|)\sum_{j\in\Z^d}\lambda_m^j\norm{X^j - Z^j}_s + C_J\sum_{j\in\Z^d,k\in V_m}\lambda_m^j\norm{X^{j+k}-Z^{j+k}}_s\\ &+ \big(\sum_{j\in\Z^d}\lambda_m^j\norm{Z^j}^2_s\big)^{\frac{1}{2}} \big(\sum_{j\in\Z^d}\lambda_m^j \big(\sum_{k\in V_m}d_{\mathcal{E}}(\omega^{j,k},\beta^{j,k})\big)^2\big)^{\frac{1}{2}} \\&+  \big(\sum_{j\in\Z^d,k\in V_m}\lambda_m^j\norm{Z^{j+k}}^2_s\big)^{\frac{1}{2}}\big(\sum_{j\in\Z^d,k\in V_m}\lambda_m^j d_{\mathcal{E}}(\omega^{j,k},\beta^{j,k})^2\big)^{\frac{1}{2}}\\
&+ \sum_{j\in \Z^d,k\in V_n - V_m}\lambda_m^j\snorm{\beta^{j,k}}\big(\norm{Z^j}_s +1\big)\bigg)ds\\ 
&\leq  2\sum_{j\in \Z^d}\lambda_m^j\norm{Q^j - R^j}_T\\&+ \int_0^t \bigg[ (C+ C_J |V_m|)\sum_{j\in\Z^d}\lambda_m^j\norm{X^j - Z^j}_s + \rho C_J|V_m|\sum_{j\in\Z^d}\lambda_m^j\norm{X^{j}-Z^{j}}_s\\ &+ |V_m|^{\frac{1}{2}}\big(\sum_{j\in\Z^d}\lambda_m^j\norm{Z^j}^2_s\big)^{\frac{1}{2}} \big(\sum_{j\in \Z^d,k\in V_m}\lambda_m^j d_{\mathcal{E}}(\omega^{j,k},\beta^{j,k})^2\big)^{\frac{1}{2}} \\&+  \sqrt{\rho}|V_m|^{\frac{1}{2}}\big(\sum_{j\in\Z^d}\lambda_m^j\norm{Z^j}^2_s\big)^{\frac{1}{2}}\big(\sum_{j\in\Z^d,k\in V_m}\lambda_m^j d_{\mathcal{E}}(\omega^{j,k},\beta^{j,k})^2\big)^{\frac{1}{2}} \\
&+ \sum_{j\in \Z^d,k\in V_n - V_m}\lambda_m^j\snorm{\beta^{j,k}}\big(\norm{Z^j}_s +1\big)\bigg)ds,
\end{align*}
by \eqref{eq: lambdakappa2}, and using Jensen's Inequality to obtain the bound \newline$\big(\sum_{k\in V_m}d_{\mathcal{E}}(\omega^{j,k},\beta^{j,k})\big)^2\leq |V_m|\sum_{k\in V_m}d_{\mathcal{E}}(\omega^{j,k},\beta^{j,k})^2$. Hence by Gronwall's Inequality,
\begin{multline*}
\sum_{j\in\Z^d}\lambda_m^j\norm{X^j - Z^j}_T \leq \exp\big( TC+T(1+\rho)C_J |V_m| \big)\\ \times\bigg[\big(1+\sqrt{\rho}\big)T|V_m|^{\frac{1}{2}}\big(\sum_{j\in\Z^d}\lambda_m^j\norm{Z^j}^2_T\big)^{\frac{1}{2}} \big(\sum_{j \in \Z^d,k\in V_m}\lambda_m^j d_{\mathcal{E}}(\omega^{j,k},\beta^{j,k})^2\big)^{\frac{1}{2}} \\+ 2 \sum_{j\in \Z^d}\lambda_m^j\norm{Q^j - R^j}_T 
+ T\sum_{j\in \Z^d,k\in V_n - V_m}\lambda_m^j\snorm{\beta^{j,k}}\big(\norm{Z^j}_T + 1\big)\bigg]. 
\end{multline*}
\end{proof}
\begin{lemma}\label{Lemma: bound exp Z}
Suppose that $\mathcal{Z} = \big(R^j,\omega^{j,k}\big)_{j,k\in\Z^d} \in \bar{\T}^{\Z^d}$ and $\Psi^m(\mathcal{Z})$ exists. Then
\begin{multline*}
\norm{\Psi^m(\mathcal{Z})^j}_T \leq \exp\bigg(T\big(C + 2\sum_{k\in V_m}\snorm{\omega^{j,k}} \big) \bigg) \\+ 2\exp\bigg(T\big(C + \sum_{k\in V_m}\snorm{\omega^{j,k}} \big) \bigg)\norm{R^j}_T.
\end{multline*}
\end{lemma}
\begin{proof}
Suppose that $[\tau,\gamma] \subseteq [0,T]$ is such that $\Psi^m(\mathcal{Z})^j_t \geq 0$ for all $t\in [\tau,\gamma]$. Then, making use of Assumption \ref{Assumption Absolute Bound}, for $t\in [\tau,\gamma]$,
\begin{multline*}
|\Psi^m(\mathcal{Z})^j_t| \leq |\Psi^m(\mathcal{Z})_\tau^j|+ R^j_t - R^j_\tau \\  + \int_\tau^t \bigg(C\norm{\Psi^m(\mathcal{Z})^j}_s + \sum_{k\in V_m}\snorm{\omega^{j,k}}\big(1+\norm{\Psi^m(\mathcal{Z})^j}_s\big)\bigg)ds .
\end{multline*}
We take $\tau$ to be such that, either $\tau = 0$ (in which case $\Psi^m(\mathcal{Z})^j_0 = R^j_0$) or $\Psi^m(\mathcal{Z})^j_\tau = 0$. Whichever is the case, we find that
\begin{multline*}
|\Psi^m(\mathcal{Z})_t^j| \leq \int_0^t \bigg(C\norm{\Psi^m(\mathcal{Z})^j}_s + \sum_{k\in V_m}\snorm{\omega^{j,k}}\big(1+\norm{\Psi^m(\mathcal{Z})^j}_s\big)\bigg)ds + 2\norm{R^j}_t.
\end{multline*}
We would obtain the same inequality if we had assumed that $\Psi^m(\mathcal{Z})^j_t \leq 0$ for all $t\in [\tau,\gamma]$. Since any $t\in [0,T]$ must satisfy either $\Psi^m(\mathcal{Z})^j_t \geq 0$ or $\Psi^m(\mathcal{Z})^j_t \leq 0$, we find that for all $t\in [0,T]$,
\begin{multline*}
\norm{\Psi^m(\mathcal{Z})^j}_t \leq \int_0^t \bigg(C\norm{\Psi^m(\mathcal{Z})^j}_s + \sum_{k\in V_m}\snorm{\omega^{j,k}}\big(1+\norm{\Psi^m(\mathcal{Z})^j}_s\big)\bigg)ds + 2\norm{R^j}_t.
\end{multline*}
By Gronwall's Inequality,
\begin{equation*}
\norm{\Psi^m(\mathcal{Z})^j}_T \leq \exp\bigg(T\big(C + \sum_{k\in V_m}\snorm{\omega^{j,k}} \big) \bigg)\bigg(2\norm{R^j}_T + T\sum_{k\in V_m}\snorm{\omega^{j,k}} \bigg).
\end{equation*}
The lemma follows directly from this.
\end{proof}
The following lemma is needed to prove Lemma \ref{Lem Uniform Bound mu Ac}.
\begin{lemma}\label{Lem: Lower bound on lambda jm}
For any $j\in \Z^d$,
\[
\linf{m}|V_m| \lambda^j_m \geq \frac{\rho-1}{\rho^2}.
\]
\end{lemma}
\begin{proof}
We first bound $h$ using \eqref{defn h}. Observe that, since $\tilde{\kappa}_m(\theta) \leq |V_m|$,
\begin{equation*}
\frac{1}{(2\pi)^d}\int_{[-\pi,\pi]^d}\bigg(\rho|V_m| - \tilde{\kappa}_m(\theta) \bigg)^{-1}d\theta \leq \frac{1}{(2\pi)^d(\rho-1)}\int_{[-\pi,\pi]^d}|V_m|^{-1}d\theta = \frac{1}{(\rho-1)|V_m|}.
\end{equation*}
This means that $h\geq |V_m|(\rho-1)$. We use a Taylor expansion to bound $\lambda^j_m$ as follows,
\begin{align}
\lambda^j_m &= \frac{h}{(2\pi)^d}\int_{[-\pi,\pi]^d}\exp\big(i\langle j,\theta\rangle \big)\bigg(\rho|V_m| - \tilde{\kappa}_m(\theta) \bigg)^{-1}d\theta\nonumber \\
&= \frac{h}{\rho|V_m|(2\pi)^d}\sum_{k=0}^\infty \int_{[-\pi,\pi]^d}\exp\big(i\langle j,\theta\rangle \big)\bigg(\frac{\tilde{\kappa}_m(\theta)}{\rho|V_m|} \bigg)^k d\theta\nonumber\\
&\geq \frac{h}{\rho|V_m|(2\pi)^d} \int_{[-\pi,\pi]^d}\exp\big(i\langle j,\theta\rangle \big)\bigg(\frac{\tilde{\kappa}_m(\theta)}{\rho|V_m|} \bigg) d\theta.\label{eq:h tmp}
\end{align}
This last step is due to the fact that for any $k\in\Z^+$
\[
\int_{[-\pi,\pi]^d}\exp\big(i\langle j,\theta\rangle \big)\bigg(\frac{\tilde{\kappa}_m(\theta)}{\rho|V_m|} \bigg)^kd\theta \geq 0.
\]
For example, if $k=2$ in the above, then from the convolution formula for Fourier Series,
\begin{align*}
\frac{1}{\rho|V_m|(2\pi)^d}\int_{[-\pi,\pi]^d}\exp\big(i\langle j,\theta\rangle \big)\bigg(\frac{\tilde{\kappa}_m(\theta)}{\rho|V_m|} \bigg)^k d\theta = \frac{1}{\rho^3|V_m|^3}\sum_{r\in\Z^d} \kappa_m^{j-r}\kappa_m^r.
\end{align*}
This is non-negative because each $\kappa_m^r$ is non-negative. This result easily generalises to $k>2$. Coming back to \eqref{eq:h tmp}, we see that
\begin{equation}
\lambda^j_m \geq \frac{h}{\rho|V_m|(2\pi)^d}(2\pi)^d \frac{\kappa^j_m}{\rho|V_m|} \geq \frac{|V_m|(\rho-1)}{\rho^2|V_m|^2}\kappa^j_m,
\end{equation}
since $h\geq |V_m|(\rho-1)$. Once $m$ is large enough that $j\in V_m$, we have the lemma.
\end{proof}

\section{Lemmas Auxiliary to the Proof of Theorem \ref{Theorem Main LDP}}
The three main results of this Section are Lemmas \ref{Lem Bound hat mu in Ac}, \ref{Lem Uniform Bound mu Ac} and \ref{Lem alpha c relation}. They are all needed to complete the proof of Theorem \ref{Theorem Main LDP}. For the following lemma recall that $\mathcal{A}_c$ is defined in \eqref{defn Ac}.
\begin{lemma}\label{Lem Bound hat mu in Ac}
\begin{equation}
\lim_{c\to\infty}\lsup{n}\frac{1}{|V_n|}\log\mathbb{P}\bigg(\hat{\mu}^n\big(\mathcal{Y}^n\big) \notin \mathcal{A}_c \bigg) = -\infty.
\end{equation}
\end{lemma}
\begin{proof}
For any $m > \mathfrak{m}_0$,
\begin{multline*}
\mathbb{P}\bigg(\hat{\mu}^n\big(\mathcal{Y}^n\big) \notin \mathcal{A}_c \bigg) \leq \mathbb{P}\bigg(\sum_{j\in V_n}\norm{W^{n,j}}_T^2 > c|V_n| \bigg)+\\
\mathbb{P}\bigg(|V_m|^{2+2\rho}\exp\big[(4+2\rho)TC_J|V_m|\big]  \sum_{j\in V_n}\big(\sum_{k\notin V_m}\snorm{J^{n,j,k}} \big)^2   > c|V_n| \bigg)\\
+\mathbb{P}\bigg(|V_m|^{1+\rho} \exp\big[(3+\rho)TC_J|V_m| \big]\sum_{j\in V_n} \sum_{k\notin V_m}\snorm{J^{n,j,k}}  > c|V_n|\bigg).
\end{multline*}
We bound these two terms using the exponential Chebyshev Inequality. Using Assumption \ref{Assumption Noise LDP},
\begin{align*}
\mathbb{P}\bigg(\sum_{j\in V_n}\norm{W^{n,j}}_T^2 > c|V_n| \bigg) = \mathbb{P}\bigg(\mathfrak{e}_1\sum_{j\in V_n}\norm{W^{n,j}}_T^2 > \mathfrak{e}_1c|V_n| \bigg) \\
\leq \exp\big(-\mathfrak{e}_1 c|V_n| \big)\Exp\left[\exp\big(\mathfrak{e}_1\sum_{j\in V_n}\norm{W^{n,j}}_T^2 \big) \right] \\
\leq \exp\big(-\mathfrak{e}_1 c|V_n|  + \mathfrak{e}_2|V_n|\big).
\end{align*}
Making similar use of Assumption \ref{Assumption Connection Growth} , we find that
\begin{multline*}
\mathbb{P}\bigg(|V_m|^{2+2\rho}\exp\big[(4+2\rho)TC_J|V_m|\big] \sum_{j\in V_n}\big(\sum_{k\in V_m}\snorm{J^{n,j,k}} \big)^2  > c|V_n|\bigg)\\ \leq \exp\big(-c\mathfrak{a}_1|V_n|+ \mathfrak{a}_2|V_n| \big) 
\end{multline*}
\begin{multline*}
\mathbb{P}\bigg(|V_m|^{1+\rho} \exp\big[(3+\rho)TC_J|V_m| \big]\sum_{j\in V_n,k\in V_m}\snorm{J^{n,j,k}}  > c|V_n|\bigg)\\ \leq \exp\big(-c\mathfrak{a}_1|V_n|+\mathfrak{a}_2|V_n| \big) .
\end{multline*}
We thus see that
\begin{equation*}
\lsup{n}\frac{1}{|V_n|}\log\mathbb{P}\bigg(\hat{\mu}^n\big(\mathcal{Y}^n\big) \notin \mathcal{A}_c \bigg) \leq \max\left\lbrace - c\mathfrak{a}_1 + \mathfrak{a}_2,-\mathfrak{e}_1c + \mathfrak{e}_2 \right\rbrace.
\end{equation*}
The lemma follows directly from the above.
\end{proof}
\begin{lemma}\label{Lem Uniform Bound mu Ac}
For any $c,\delta > 0$, there exists an $m\in \Z^+$ such that
\begin{equation}\label{eq BC first result}
\sup_{n\geq m}\sup_{\mu\in\mathcal{A}_c}d^{\mathcal{P}}\big(\tilde{\Psi}^m(\mu) ,\tilde{\Psi}^n(\mu) \big) \leq \delta.
\end{equation}
For all $\mu \in \cup_{c\geq 0}\mathcal{A}_c$, 
\begin{equation}\label{eq mu Psi inverse}
\mu \circ \Psi^{-1} = \lim_{m\to \infty}\tilde{\Psi}^m(\mu).
\end{equation}
\end{lemma}
\begin{proof}
Let $r\in \Z^+$ be such that $2^{-r} \leq \frac{\delta}{2}$. It follows from the definition of the metric $d^{\mathcal{P}}$ in \eqref{defn prohorov metric T} that
\[
d^{\mathcal{P}}\big(\tilde{\Psi}^m(\mu) ,\tilde{\Psi}^n(\mu) \big) \leq 2^{-r} + d_r^{\mathcal{P}}\big(\tilde{\Psi}^m(\mu) ,\tilde{\Psi}^n(\mu) \big). 
\]
For \eqref{eq BC first result}, it thus suffices for us to show that for fixed $r\in\Z^+$ and $\delta > 0$, there exists an $m\in \Z^+$ such that
\begin{equation}\label{Lemma Intermediate Technical 0}
\sup_{n\geq m}\sup_{\mu\in\mathcal{A}_c}d^{\mathcal{P}}_r\big(\tilde{\Psi}^m(\mu) ,\tilde{\Psi}^n(\mu) \big)\leq \delta.
\end{equation}
Let $\mathcal{X}^j = \big( Q^j,(\omega^{j,k})_{k\in\Z^d}\big)$, $X = \Psi^m(\mathcal{X})$ and $Z = \Psi^n(\mathcal{X})$. By Lemma \ref{Lemma: Bound XZ 1},
\begin{multline*}
\sum_{j\in\Z^d}\lambda^j_m\norm{X^j - Z^j}_T \leq T\exp\big(T(2C+(1+\rho)C_J|V_m|) \big)\times \\ 
\sum_{j\in \Z^d,k\in V_n - V_m}\lambda_m^j\snorm{\omega^{j,k}}\big(\norm{Z^j}_T + 1\big).
\end{multline*}
We multiply both sides of the above equation by $|V_m|$ and use Lemma \ref{Lemma: bound exp Z}, finding that
\begin{multline*}
|V_m|\sum_{j\in\Z^d}\lambda^j_m\norm{X^j - Z^j}_T \leq |V_m|T\exp\big(2TC\big)\times \\ \sum_{j\in \Z^d,k\in V_n - V_m}\lambda^j_m\snorm{\omega^{j,k}}\bigg[ \exp\big((3+\rho)TC_J |V_m| \big)\\+ \exp\big((2+\rho)TC_J |V_m| \big)\norm{Q^j}_T +\exp\big((1+\rho)TC_J|V_m|\big)\bigg].
\end{multline*}
Let $m$ be sufficiently large that
\[
\sup_{j\in V_r}|V_m|\lambda^j_m \geq \frac{\rho-1}{\rho^2},
\]
which is possible thanks to Lemma \ref{Lem: Lower bound on lambda jm}. We then use Chebyshev's Inequality to find that
\begin{align}
\mu\bigg(\sum_{j\in V_r}\norm{X^j-Z^j}_T > \epsilon \bigg) \leq \mu\bigg( \frac{\rho^2}{\rho-1}|V_m|\sum_{j\in\Z^d}\lambda^j_m\norm{X^j - Z^j}_T >  \epsilon\bigg)\nonumber\\
\leq  \frac{\rho^2}{\rho-1}\epsilon^{-1}|V_m|T\exp\big(2TC\big)\Exp^{\mu}\bigg[ 2\exp\big((3+\rho)TC_J |V_m| \big)\sum_{j\in \Z^d,k\in V_n - V_m}\lambda^j_m\snorm{\omega^{j,k}}\nonumber \\+\exp\big((2+\rho)TC_J |V_m| \big)\sum_{j\in \Z^d,k\in V_n - V_m}\lambda^j_m\snorm{\omega^{j,k}}\norm{Q^j}_T\bigg] \nonumber\\
\leq  \frac{2\rho^2}{\rho-1}c\epsilon^{-1}T\exp\big(2TC\big)|V_m|^{-\rho} +  \frac{\rho^2}{\rho-1}\epsilon^{-1}|V_m|T\exp\big(2TC+(2+\rho)TC_J |V_m|\big)\times\nonumber \\ \Exp^{\mu}\bigg[\sum_{j\in \Z^d}\lambda^j_m\big(\sum_{k \in V_n - V_m}\snorm{\omega^{j,k}} \big)^2\bigg]^{\frac{1}{2}}\times\Exp^{\mu}\bigg[\sum_{j\in\Z^d}\lambda^j_m\norm{Q^j}_T^2 \bigg]^{\frac{1}{2}}\nonumber \\
\leq \frac{2\rho^2c}{\rho-1}\epsilon^{-1}T\exp\big(2TC\big)|V_m|^{-\rho} + \frac{\rho^2}{\rho-1}c\epsilon^{-1}|V_m|^{-\rho}T\exp\big(2TC \big),\label{eq: mu convergence temp}
\end{align}
where we have used the Cauchy-Schwarz Inequality, and twice used the definition of $\mathcal{A}_c$ in \eqref{defn Ac}. Since $\rho > 1$, for $m$ sufficiently large the above equation is less than $\delta$, yielding \eqref{eq BC first result}. We now prove \eqref{eq mu Psi inverse}. Let $\lbrace p_k \rbrace_{k=1}^\infty \subset \Z^+$ be such that $p_k > k$ and
\[
\mu\bigg(\sum_{j\in V_r}\norm{\Psi^{p_k}(\mathcal{X})^j-\Psi^{p_{k+1}}(\mathcal{X})^j}_T > 2^{-k} \bigg) \leq 2^{-k}.
\]
It follows from the Borel-Cantelli Lemma and the two previous equations that $\lbrace\Psi^{p_k}(\mathcal{X})\rbrace_{k=1}^\infty$ converges almost surely. Let $Y^j = \lim_{k\to\infty}\Psi^{p_k}(\mathcal{X})^j$ and $\epsilon_m = |V_m|^{\frac{1}{2}(1-\rho)}$. Then it follows from \eqref{eq: mu convergence temp} that, since $\rho > 1$, 
\begin{equation}
\sum_{m=1}^\infty \mu\bigg( \sum_{j\in V_r}\norm{\Psi^{m}(\mathcal{X}) - \Psi^{p_m}(\mathcal{X})} > \epsilon_m\bigg) \leq \frac{3\rho^2 c}{\rho-1}T\exp\big(2TC\big)\sum_{m=1}^\infty |V_m|^{-\frac{1+\rho}{2}} < \infty.
\end{equation}
By the Borel-Cantelli Lemma, for all $j \in V_r$, $\Psi^m(\mathcal{X})^j \to Y^j$ almost surely. It remains for us to show that $Y$ satisfies the relation in \eqref{eq:fundamentalmult 1}. It suffices to show that for all $t\in [0,T]$,
\begin{align}
\int_0^t \sum_{k\in\Z^d}\Lambda^k_s\big(\omega^{j,k},\Psi^m(\mathcal{X})^j,\Psi^m(\mathcal{X})^{j+k}  \big) ds &\to \int_0^t \sum_{k\in\Z^d}\Lambda^k_s\big(\omega^{j,k},Y^j,Y^{j+k}\big)ds\label{eq:BC1} \\
\int_0^t \mathfrak{b}\big(\Psi_s^m(\mathcal{X})^j \big) ds &\to \int_0^t \mathfrak{b}\big(Y_s^j\big) ds \label{eq:BC2}
\end{align}
It follows from an application of the Borel-Cantelli Lemma to \eqref{defn Ac} that, $\mu$ almost surely, there must exist some $q\in \Z^+$ such that $\omega^{j,k} = 0$ for all $k\notin V_q$. Since $r$ is arbitrary, we may assume that $r \geq q$. Since $\Psi^m(\mathcal{X})^{j+k}  \to Y^{j+k}$ for all $k\in V_q$, \eqref{eq:BC1} follows from Assumption \ref{Assumption NonUniform} and the dominated convergence theorem. \eqref{eq:BC2} follows from Assumption \ref{Assumption Absolute Bound} and the dominated convergence theorem.
\end{proof}
In the following, recall the definition of $\mathcal{A}_c$ in \eqref{defn Ac}.
\begin{lemma}\label{Lem alpha c relation}
For each $\alpha\in\R^+$, there exists $c\in \R^+$ such that
\begin{equation}\label{eq appropriate c}
\left\lbrace \mu \in \mathcal{P}_{\bar{S}}\big(\bar{\T}^{\Z^d} \big) | I_{\mathcal{Y}}(\mu) \leq \alpha\right\rbrace \subseteq \mathcal{A}_c.
\end{equation}
\end{lemma}
\begin{proof}
Suppose that $I_{\mathcal{Y}}(\mu) \leq \alpha$. By \cite[Theorem 4.5.10]{dembo-zeitouni:97}, for all $f\in C_b\big(\bar{\T}^{\Z^d}\big)$, where $C_b\big(\bar{\T}^{\Z^d}\big)$ is the set of all bounded continuous functions on $\bar{\T}^{\Z^d}$ which are $\mathcal{B}\big(\T^{\bar{V}_q}\big)$-measurable for some $q\in\Z^+$,
\begin{equation}
\Exp^{\mu}\left[f\right] - \lsup{n}\frac{1}{|V_n|}\log\Exp\left[\exp\bigg(\sum_{j\in V_n}f(\bar{S}^j\tilde{\mathcal{Y}}^n) \bigg) \right] \leq \alpha.\label{eq: mu alpha bound}
\end{equation}
We will show that in order that \eqref{eq appropriate c} is satisfied, it suffices to take
\begin{equation}
c = \max\left\lbrace \frac{\alpha + \mathfrak{e}_2}{\mathfrak{e}_1},\frac{\alpha + \mathfrak{a}_2}{\mathfrak{a}_1}\right\rbrace,
\end{equation}
where the constants $\mathfrak{e}_1,\mathfrak{e}_2,\mathfrak{a}_1,\mathfrak{a}_2$ are defined in Assumptions \ref{Assumption Connection Growth} and \ref{Assumption Noise LDP}. For $r\in\Z^+$, let $\phi_r(\mathcal{Y}^n) := \mathfrak{e}_1\sigma_r\big(\norm{W^{n,0}}^2_T\big)$, where $\sigma_r:\R^+ \to \R^+$ is continuous and bounded, and $\sigma_r(x) \uparrow x$ as $r\to \infty$. Substituting $f=\phi_r$ into \eqref{eq: mu alpha bound}, taking $r\to \infty$, and using Assumption \ref{Assumption Noise LDP}, we find that
\begin{equation}
\mathfrak{e}_1\Exp^{\mu}\left[\norm{X^0}_T^2\right]  \leq \alpha + \mathfrak{e}_2.\label{eq: mu alpha bound 2}
\end{equation}
This means that if $c\geq \big(\alpha + \mathfrak{e}_2\big) / \mathfrak{e}_1$, then the first condition for membership of $\mathcal{A}_c$ is satisfied (see \eqref{defn Ac}). Next, we define for $r\in\Z^+$, $r > m \geq \mathfrak{m}_0$
\begin{equation*}
\phi_r(\mathcal{Y}^n) =  \mathfrak{a}_1|V_m|^{2\rho+2}\exp\bigg(2T(2+\rho)C_J |V_m|\bigg) \bigg(\sum_{k\in V_r -V_m}\snorm{J^{n,0,k}} \bigg)^2. 
\end{equation*}
Substituting $f=\phi_r$ into \eqref{eq: mu alpha bound}, taking $r\to \infty$, and using Assumption \ref{Assumption Connection Growth} we find that
\begin{multline*}
\Exp^{\mu}\bigg[\mathfrak{a}_1|V_m|^{2\rho+2}\exp\bigg(2T(2+\rho)C_J |V_m|\bigg) \sum_{j\in V_n}\bigg(\sum_{k\notin V_m}\snorm{J^{n,j,k}} \bigg)^2\bigg] \leq \alpha + \mathfrak{a}_2.
\end{multline*}
This means that the second condition for membership of $\mathcal{A}_c$ is satisfied once $c\geq \mathfrak{a}_1^{-1}\big(\alpha+\mathfrak{a}_2\big)$. We similarly find using Assumption \ref{Assumption Connection Growth} that
\begin{equation*}
\Exp^{\mu}\bigg[\mathfrak{a}_1|V_m|^{\rho+1}\exp\bigg(T(3+\rho)C_J |V_m|\bigg) \sum_{j\in V_n}\sum_{k\notin V_m}\snorm{J^{n,j,k}}\bigg] \leq \alpha + \mathfrak{a}_2,
\end{equation*}
which means that the third condition for membership of $\mathcal{A}_c$ is satisfied once $c\geq \mathfrak{a}_1^{-1}\big(\alpha+\mathfrak{a}_2\big)$.
\end{proof}


\begin{thebibliography}{10}

\bibitem{atay2016perspectives}
{\sc F.~M. Atay, S.~Banisch, P.~Blanchard, B.~Cessac, E.~Olbrich, and
  D.~Volchenkov}, {\em Perspectives on multi-level dynamics}, arXiv preprint
  arXiv:1606.05665,  (2016).

\bibitem{baladron-fasoli-etal:12b}
{\sc J.~Baladron, D.~Fasoli, O.~Faugeras, and J.~Touboul}, {\em Mean-field
  description and propagation of chaos in networks of {H}odgkin-{H}uxley and
  {F}itzhugh-{N}agumo neurons}, The Journal of Mathematical Neuroscience, 2
  (2012).

\bibitem{bayraktar2020graphon}
{\sc E.~Bayraktar, S.~Chakraborty, and R.~Wu}, {\em Graphon mean field
  systems}, arXiv preprint arXiv:2003.13180,  (2020).

\bibitem{ben-arous-guionnet:95}
{\sc G.~Ben-Arous and A.~Guionnet}, {\em Large deviations for langevin spin
  glass dynamics}, Probability Theory and Related Fields, 102 (1995),
  pp.~455--509.

\bibitem{billingsley:99}
{\sc P.~Billingsley}, {\em Convergence of Probability Measures}, Wiley series
  in probability and statistics, 1999.

\bibitem{bossy2015clarification}
{\sc M.~Bossy, O.~Faugeras, and D.~Talay}, {\em Clarification and complement to
  ``mean-field description and propagation of chaos in networks of
  hodgkin--huxley and fitzhugh--nagumo neurons''}, The Journal of Mathematical
  Neuroscience (JMN), 5 (2015), pp.~1--23.

\bibitem{bossy-talay:97}
{\sc M.~Bossy and D.~Talay}, {\em A stochastic particle method for the
  mckean-vlasov and the burgers equation}, Mathematics of computation, 66
  (1997), pp.~157--192.

\bibitem{bressloff:09}
{\sc P.~Bressloff}, {\em Stochastic neural field theory and the system-size
  expansion}, SIAM J. Appl. Math, 70 (2009), pp.~1488--1521.

\bibitem{bressloff:12}
{\sc P.~Bressloff}, {\em Spatiotemporal dynamics of continuum neural fields},
  Journal of Physics A: Mathematical and Theoretical, 45 (2012).

\bibitem{budhiraja-dupuis-etal:12}
{\sc A.~Budhiraja, P.~Dupuis, and M.~Fischer}, {\em Large deviation properties
  of weakly interacting processes via weak convergence methods}, Annals of
  Probability, 40 (2012), pp.~74--102.

\bibitem{buice-cowan-etal:10}
{\sc M.~Buice, J.~Cowan, and C.~Chow}, {\em Systematic fluctuation expansion
  for neural network activity equations}, Neural computation, 22 (2010),
  pp.~377--426.

\bibitem{coppini2020law}
{\sc F.~Coppini, H.~Dietert, and G.~Giacomin}, {\em A law of large numbers and
  large deviations for interacting diffusions on erd{\H{o}}s--r{\'e}nyi
  graphs}, Stochastics and Dynamics, 20 (2020), p.~2050010.

\bibitem{dai1996mckean}
{\sc P.~Dai~Pra and F.~den Hollander}, {\em Mckean-vlasov limit for interacting
  random processes in random media}, Journal of statistical physics, 84 (1996),
  pp.~735--772.

\bibitem{dawson-del-moral:05}
{\sc D.~Dawson and P.~Del~Moral}, {\em Large deviations for interacting
  processes in the strong topology}, in Statistical Modeling and Analysis for
  Complex Data Problems, Springer US, 2005.

\bibitem{delarue2015particle}
{\sc F.~Delarue, J.~Inglis, S.~Rubenthaler, and E.~Tanr{\'e}}, {\em Particle
  systems with a singular mean-field self-excitation. application to neuronal
  networks}, Stochastic Processes and their Applications, 125 (2015),
  pp.~2451--2492.

\bibitem{delattre2016hawkes}
{\sc S.~Delattre, N.~Fournier, M.~Hoffmann, et~al.}, {\em Hawkes processes on
  large networks}, The Annals of Applied Probability, 26 (2016), pp.~216--261.

\bibitem{delattre2016note}
{\sc S.~Delattre, G.~Giacomin, and E.~Lu{\c{c}}on}, {\em A note on dynamical
  models on random graphs and fokker--planck equations}, Journal of Statistical
  Physics, 165 (2016), pp.~785--798.

\bibitem{dembo-zeitouni:97}
{\sc A.~Dembo and O.~Zeitouni}, {\em Large deviations techniques}, Springer,
  1997.
\newblock 2nd Edition.

\bibitem{destexhe-mainen-etal:94}
{\sc A.~Destexhe, Z.~Mainen, and T.~Sejnowski}, {\em An efficient method for
  computing synaptic conductances based on a kinetic model of receptor
  binding}, Neural Computation, 6 (1994), pp.~14---18.

\bibitem{donsker-varadhan:85}
{\sc M.~Donsker and S.~Varadhan}, {\em Large deviations for stationary
  {G}aussian processes}, Commun. Math. Phys., 97 (1985), pp.~187--210.

\bibitem{dugundji1951extension}
{\sc J.~Dugundji}, {\em An extension of tietze's theorem}, Pacific J. Math, 1
  (1951), pp.~353--367.

\bibitem{ellis:85}
{\sc R.~Ellis}, {\em Entropy, large deviations and statistical mechanics},
  Springer, 1985.

\bibitem{ermentrout-terman:10}
{\sc G.~B. Ermentrout and D.~Terman}, {\em Foundations of Mathematical
  Neuroscience}, Interdisciplinary Applied Mathematics, Springer, 2010.

\bibitem{fasoli2015complexity}
{\sc D.~Fasoli, A.~Cattani, and S.~Panzeri}, {\em The complexity of dynamics in
  small neural circuits}, arXiv preprint arXiv:1506.08995,  (2015).

\bibitem{e16126722}
{\sc O.~Faugeras and J.~MacLaurin}, {\em A large deviation principle and an
  expression of the rate function for a discrete stationary gaussian process},
  Entropy, 16 (2014), p.~6722.

\bibitem{faugeras-maclaurin:16}
{\sc O.~Faugeras and J.~MacLaurin}, {\em Large deviations of a
  spatially-stationary network of interacting neurons}, arxiv depot, INRIA
  Sophia Antipolis, http://arxiv.org/abs/1404.0732, 2014.

\bibitem{fischer:12}
{\sc M.~Fischer et~al.}, {\em On the form of the large deviation rate function
  for the empirical measures of weakly interacting systems}, Bernoulli, 20
  (2014), pp.~1765--1801.

\bibitem{fitzhugh:55}
{\sc R.~FitzHugh}, {\em Mathematical models of threshold phenomena in the nerve
  membrane}, Bulletin of Mathematical Biology, 17 (1955), pp.~257--278.

\bibitem{fitzhugh:66}
{\sc R.~Fitzhugh}, {\em {Theoretical Effect of Temperature on Threshold in the
  Hodgkin-Huxley Nerve Model}}, The Journal of General Physiology, 49 (1966),
  pp.~989--1005.

\bibitem{fitzhugh:69}
{\sc R.~FitzHugh}, {\em Mathematical models of excitation and propagation in
  nerve}, McGraw-Hill Companies, 1969, ch.~1.

\bibitem{galves2013infinite}
{\sc A.~Galves and E.~L{\"o}cherbach}, {\em Infinite systems of interacting
  chains with memory of variable length a stochastic model for biological
  neural nets}, Journal of Statistical Physics, 151 (2013), pp.~896--921.

\bibitem{georgii:88}
{\sc H.-O. Georgii}, {\em Gibbs Measures and Phase Transitions}, De Gruyter,
  1988.

\bibitem{gerstner-kistler:02b}
{\sc W.~Gerstner and W.~Kistler}, {\em Spiking Neuron Models}, Cambridge
  University Press, 2002.

\bibitem{goldys:01}
{\sc B.~Goldys and M.~Musiela}, {\em Infinite dimensional diffusions,
  kolmogorov equations and interest rate models}, Option Pricing, Interest
  Rates and Risk Management,  (2001), pp.~314--335.

\bibitem{nagumo-arimoto-etal:62}
{\sc J.~Nagumo, S.~Arimoto, and S.~Yoshizawa}, {\em An active pulse
  transmission line simulating nerve axon}, Proceedings of the IRE, 50 (1962),
  pp.~2061--2070.

\bibitem{robins2007introduction}
{\sc G.~Robins, P.~Pattison, Y.~Kalish, and D.~Lusher}, {\em An introduction to
  exponential random graph (p*) models for social networks}, Social networks,
  29 (2007), pp.~173--191.

\bibitem{shiga-shimizu:80}
{\sc T.~Shiga and A.~Shimizu}, {\em Infinite dimensional stochastic
  differential equations and their applications}, Journal Mathematics Kyoto
  University, 20 (1980), pp.~395--416.

\bibitem{sporns2011networks}
{\sc O.~Sporns}, {\em Networks of the Brain}, MIT press, 2011.

\bibitem{sznitman:91}
{\sc A.-S. Sznitman}, {\em Topics in propagation of chaos}, in Ecole d'Et{\'e}
  de Probabilit{\'e}s de Saint-Flour XIX --- 1989, D.~Burkholder, E.~Pardoux,
  and A.-S. Sznitman, eds., vol.~1464 of Lecture Notes in Mathematics, Springer
  Berlin / Heidelberg, 1991, pp.~165--251.
\newblock 10.1007/BFb0085169.

\end{thebibliography}
\end{document}